\documentclass[a4paper,11pt]{amsart}
\usepackage[T1]{fontenc}
\usepackage[utf8]{inputenc}
\usepackage{lmodern}
\usepackage{enumerate}
\usepackage{amssymb,amsxtra}
\usepackage[all]{xy}
\usepackage{xcolor}
\usepackage{nicefrac,mathtools}
\usepackage{microtype}
\usepackage{amscd}
\usepackage{geometry}
\usepackage{mathbbol}
\usepackage{comment}

\usepackage[pdftitle={...},
 pdfauthor={Dami\'an Ferraro},
 pdfsubject={Mathematics}]{hyperref}
\usepackage{cite}

\numberwithin{equation}{section}
\theoremstyle{plain}
\newtheorem{theorem}[equation]{Theorem}
\newtheorem{lemma}[equation]{Lemma}
\newtheorem{proposition}[equation]{Proposition}
\newtheorem{corollary}[equation]{Corollary}
\theoremstyle{definition}
\newtheorem{definition}[equation]{Definition}

\theoremstyle{remark}
\newtheorem{remark}[equation]{Remark}

\newcommand{\bB}{\mathbb{B}}
\newcommand{\bC}{\mathbb{C}}

\newcommand{\bF}{\mathbb{F}}

\newcommand{\bN}{\mathbb{N}}
\newcommand{\bR}{\mathbb{R}}

\newcommand{\cB}{\mathcal{B}}
\newcommand{\cC}{\mathcal{C}}

\newcommand{\cF}{\mathcal{F}}

\newcommand{\cT}{\mathcal{T}}

\newcommand{\cX}{\mathcal{X}}

\newcommand{\intform}[1]{\tilde{ #1 } }

\newcommand{\lt}{\mathtt{lt}}

\newcommand{\regrep}[1]{\Lambda^{#1}}
\newcommand{\regrepEN}[1]{\lambda_{#1}}
\newcommand{\rmu}{{r^{-1}}}

\newcommand{\smu}{{s^{-1}}}

\newcommand{\tmu}{{t^{-1}}}
\newcommand{\trace}{\operatorname{tr}}

\providecommand{\cspn}{\overline{\mathop{\rm span}}}

\providecommand{\Ind}{{\mathop{\rm Ind}}}
\providecommand{\red}{{\mathop{\rm r}}}
\providecommand{\spn}{{\mathop{\rm span}}}
\providecommand{\supp}{{\mathop{\rm supp}}}

\date{\today}

\begin{document}

\title{Cross-sectional C*-algebras associated to subgroups}
\author{Dami\'{a}n Ferraro}
\email{dferraro@litoralnorte.udelar.edu.uy}
\address{Departamento de Matemática y Estadística del Litoral, CENUR
	Litoral Norte, Universidad de la República,  Uruguay}
\date{\today}

\begin{abstract}
Given a Fell bundle $\cB=\{B_t\}_{t\in G}$ over a locally compact and Hausdorff group $G$ and a closed subgroup $H\subset G,$ we construct quotients $C^*_{H\uparrow \cB}(\cB)$ and $C^*_{H\uparrow G}(\cB)$ of the full cross-sectional C*-algebra $C^*(\cB)$ analogous to Exel-Ng's reduced algebras $C^*_\red(\cB)\equiv C^*_{\{e\}\uparrow \cB}(\cB)$ and $C^*_R(\cB)\equiv C^*_{\{e\}\uparrow G}(\cB).$
An absorption principle, similar to Fell's one, is used to give conditions on $\cB$ and $H$ (e.g. $G$ discrete and $\cB$ saturated, or $H$ normal) ensuring $C^*_{H\uparrow \cB}(\cB)=C^*_{H\uparrow G}(\cB).$
The tools developed here enable us to show that if the normalizer of $H$ is open in $G$ and $\cB_H:=\{B_t\}_{t\in H}$ is the reduction of $\cB$ to $H,$ then $C^*(\cB_H)=C^*_\red(\cB_H)$ if and only if $C^*_{H\uparrow \cB}(\cB)=C^*_\red(\cB);$ the last identification being implied by $C^*(\cB)=C^*_\red(\cB).$
We also prove that if $G$ is inner amenable and $C^*_\red(\cB)\otimes_{\max} C^*_\red(G)=C^*_\red(\cB)\otimes C^*_\red(G),$ then $C^*(\cB)=C^*_\red(\cB).$
\end{abstract}

\maketitle

\section{Introduction}

In the 50's and early 60's Mackey and Blattner \cite{Mk52,blattner1961induced} described an induction process $V \rightsquigarrow \Ind_H^G(V)$ that creates a unitary representation $\Ind_H^G(V)$ of a locally compact and Hausdorff (LCH) group $G$ out of a unitary representation $V$ of a closed subgroup $H\subset G$ (we write $H\leqslant G$).

When translated into the language of C*-algebras, $V \rightsquigarrow \Ind_H^G(V)$ becomes an  induction process $\pi \rightsquigarrow \Ind_{C^*(H)}^{C^*(G)}(\pi)$ for  representations of (full) group C*-algebras.
Rieffel noticed this in \cite{Rf74}, where he presented an (abstract) induction process for representations of C*-algebras he used to describe $\pi \rightsquigarrow \Ind_{C^*(H)}^{C^*(G)}(\pi).$
Later, in \cite{Fell}, Fell presented two induction theories: one for Banach *-algebraic bundles and one for *-algebras; generalizing the works of Mackey-Blattner and Rieffel (respectively).
Most of what we need is contained in Fell's original notes, but we prefer to use the standard reference \cite{FlDr88}.

Fell had a problem the other authors did not: not all the representations can be induced.
This led him to the concept of \textit{inducible} representation of a *-algebra and a notion of \textit{positivity} for representations of Banach *-algebraic bundles.
The respective definitions themselves \cite[XI 4.9 \& 8.6]{FlDr88} reveal that Fell's theories are related in the same way Mackey, Blattner and Rieffel's are.
This is made explicit in \cite[XI 9.26 \& XI 10]{FlDr88}.

We make use of all definitions and results of \cite{FlDr88} (e.g. integration, weak equivalence and weak containment of *-representations).
All the groups considered in this work are LCH and we abbreviate ``*-representation'' and ``unitary representation'' to representation.
Whenever $\cX=\{X_t\}_{t\in P}$ is a Banach bundle, $P$ stands for the base space and $X_t$ for the fiber over $t\in P.$
The set of continuous cross-sections of with compact support will be denoted $C_c(\cX).$

Fell's absorption principle states that given a saturated Banach *-algebraic bundle $\cB=\{B_t\}_{t\in G};$ $H\leqslant G;$ a non degenerate representation $S$ of $\cB$ and a unitary representation $U$ of $H,$ it follows that the representation $S|_{\cB_H}\otimes U $ of the reduction $\cB_H:=\{B_t\}_{t\in H}$ is $\cB-$positive and the respective induced representation of $\cB$ is unitary equivalent to $S\otimes \Ind_H^G(U),$ which we write
\begin{equation}\label{equ:Fell abst prin}
S\otimes \Ind_H^G(U)\cong \Ind_H^\cB(S|_{\cB_H}\otimes U).
\end{equation}
For $H=\{e\}$ ($e$ being the unit of $G$) and $U\colon H\to \bC$ trivial, $\Ind_H^G(U)$ is the left regular representation $\lambda\colon G\to \bB(L^2(G))$ and we get $S\otimes \lambda \cong \Ind_{\{e\}}^\cB(S|_{B_e}).$

The representations of the form $S\otimes \lambda$ are explicitly considered by Exel and Ng in \cite{ExNg} and, as we shall see later, those of the form $\Ind_{\{e\}}^\cB(\phi)$ appear as $ \regrepEN{\cB}\otimes_\phi 1,$ with $\lambda_\cB\colon C^*(\cB)\to \bB(L^2_e(\cB))$ being the reduced representation of \cite[Definition 2.7]{ExNg}.
Exel and Ng used those two families of representations to give natural definitions of the reduced cross-sectional C*-algebra of a Fell bundle, $C^*_\red(\cB)$ and $C^*_R(\cB)$, which turn up to be isomorphic \cite[Theorem 2.14]{ExNg}.
Exel-Ng's ideas go back to Raeburn's work on coactions \cite{Raeburn1992Crossed}, as pointed out in \cite[p 515]{ExNg}.

The construction of $C^*_R(\cB)$ can be extended considerably by using the ideas of \cite{KaLAQu2013} of producing exotic crossed products for C*-dynamical systems out of quotients $Q$ of the full group C*-algebra $C^*(G).$
Say we have a LCH group $G$ and a quotient map $q\colon C^*(G)\to Q.$
Let $\cB=\{B_t\}_{t\in G}$ be a Fell bundle and take faithful non degenerate representations $\pi\colon C^*(\cB)\to \bB(X)$ and $\rho\colon Q\to \bB(Y),$ so $\pi\otimes\rho\colon C^*(\cB)\otimes Q\to \bB(X\otimes Y)$ is faithful and non degenerate\footnote{We are considering minimal tensor products.}.
Both $\pi$ and $\rho\circ q$ are the integrated forms (or disintegrate to) representations $S\colon \cB\to \bB(X)$ and $U\colon G\to \bB(Y).$
For the canonical representation $\iota^q\colon G\to M(Q)$ we have $[S\otimes U]_b (\pi(f)\otimes \rho(z)) = \pi(bf)\otimes \rho(\iota^q(t)z),$ for all $b\in B_t,$ $f\in C_c(\cB)$ and $z\in Q.$
Hence, the image of the integrated form\footnote{The integrated form of a representation $S$ in indicated by adding $\tilde{ }$ somewhere inside or over the expression $S.$} $S\intform{\otimes}U\colon C^*(\cB)\to \bB(X\otimes Y)$ (of $S\otimes U$) is contained in $M(C^*(\cB)\otimes Q)$ and 
$\delta^{\cB q}:=S \intform{\otimes} U\colon C^*(\cB)\to M(C^*(\cB)\otimes Q)$ is independent of $(\pi,\rho).$
We define the \textit{$q-$cross-sectional C*-algebra} $C^*_q(\cB):=\delta^{\cB q}(C^*(\cB)).$
If $\cB$ is the semidirect product bundle of a C*-dynamical system $(A,G,\alpha),$ $C^*_q(\cB)\equiv A\rtimes_{q\alpha}G$ is a quotient of the full crossed product $C^*(\cB)=A\rtimes_\alpha G.$

Take for example the integrated form $ \intform{\lambda}\colon C^*(G)\to C^*_\red(G)$ of the left regular representation $\lambda\colon G\to \bB(L^2(G)).$
By definition, $C^*_{\intform{\lambda}}(\cB)=C^*_R(\cB).$
After recalling $\lambda$ is induced by the trivial representation of $\{e\},$ it is natural to replace  $\{e\}$ with any other subgroup $H\leqslant G$ and consider $q\colon C^*(G)\to C^*(G)/I,$ with $I$ the intersection of the kernels of all integrated forms of representations induced from $H.$
This associates an \textit{$H-$cross-sectional C*-algebra} $C^*_{H\uparrow G}(\cB):=C^*_q(\cB)$ to every Fell bundle $\cB$ over $G.$

There is still another natural quotient $C^*_{H\uparrow \cB}(\cB):=C^*(\cB)/J$ one may associate to $H.$
Take for $J$ the intersection of all the kernels of integrated forms of representations of $\cB$ induced from $\cB-$positive representations of $\cB_H.$
In case $\cB$ is saturated, all representations of $\cB_H$ count \cite[p 1159]{FlDr88}.
This is also the case if $H=\{e\}.$

We claim $C^*_{\{e\}\uparrow \cB}(\cB)=C^*_\red(\cB).$
Indeed, given a representation $\phi$ of $B_e\equiv \cB_H,$ the abstractly induced representation $\Ind_{B_e\uparrow C_c(\cB)}^p(\phi)$ of \cite{FlDr88} is $(\lambda_\cB\otimes_\phi 1)|_{C_c(\cB)}.$
By \cite[XI 9.26]{FlDr88}, $\lambda_\cB\otimes_\phi 1$ is the integrated form  $\intform{\Ind}_{\{e\}}^\cB(\phi).$
Hence, for all $f\in C^*(\cB) $ we have $\|\lambda_\cB(f)\|\geq \| \lambda_\cB(f)\otimes_\phi 1 \|=\|\intform{\Ind}_{\{e\}}^\cB(\phi)_f\|$  with equality if $\phi$ is faithful.
It is then clear that $J=\ker(\regrepEN{\cB})$ and $C^*_{\{e\}\uparrow \cB}(\cB)=\lambda_\cB(C^*(\cB))\equiv C^*_\red(\cB).$

Fell's absorption principle comes into play when we want to compare $C^*_{H\uparrow \cB}(\cB)$ and $C^*_{H\uparrow G}(\cB),$ for a general $H\leqslant G.$
To start with, we have canonical quotient maps
\begin{align}\label{equ: quotient maps}
\xymatrix{ C^*(\cB) \ar[r]^{q^{H\uparrow \cB}} &  C^*_{H\uparrow\cB}(\cB) }&  & \xymatrix{ C^*(\cB) \ar[r]^{q^{H\uparrow G}_\cB} & C^*_{H\uparrow G}(\cB)}
\end{align}

Assume $\cB$ is \textbf{saturated} and take faithful representations $\intform{S}$ and $\intform{U}$ of $C^*(\cB)$ and $C^*(H),$ respectively.
By \cite[XI 12.4]{FlDr88}, $\intform{\Ind}_H^G(U)$ factors through a faithful representation of $C^*(G)/I_H$ and our construction of $C^*_{H\uparrow G}(\cB)$ implies $\| q^{H\uparrow G}_\cB(f) \| = \|[S\intform{\otimes} U]_f\|$ for all $f\in C^*(\cB).$ 
By Fell's absorption principle, $\|[S\intform{\otimes }U]_f\| = \|\intform{\Ind}_H^\cB(S|_{\cB_H}\otimes U)_f\|\leq \| q^{H\uparrow \cB}(f) \|.$
This implies $\| q^{H\uparrow G}_\cB(f) \|\leq \| q^{H\uparrow \cB}(f) \|.$
In other words, there exists a unique quotient map $\pi_H^{\cB}$ making of
\begin{equation}\label{equ: pi uparrow H}
\xymatrix{  &C^*(\cB)\ar[dl]_{q^{H\uparrow \cB}}\ar[dr]^{q^{H\uparrow G}_\cB} & \\ C^*_{H\uparrow \cB}(\cB) \ar[rr]_{\pi_H^{\cB} } & & C^*_{H\uparrow G}(\cB) }
\end{equation}
a commutative diagram.

One of the main results of \cite{ExNg} is that $\pi_{\{e\}}^{\cB}\colon C^*_\red(\cB)\to C^*_R(\cB)$ is an isomorphism, even if $\cB$ is not saturated (in which case the mere existence of $\pi_{\{e\}}^{\cB}$ is in question).
When $\cB$ is saturated and $H$ is normal we can get the same conclusion  out of \eqref{equ:Fell abst prin}.
Indeed, let $T$ be a representation of $\cB_H$ with faithful integrated form.
By \cite[XI 12.8]{FlDr88}, $T\preceq \Ind_H^\cB(T)|_{\cB_H}\preceq S|_{\cB_H}$ and this implies $\intform{S|_{\cB_H}}$ is faithful.
The continuity of the induction process with respect to the regional topology \cite[XI 12.4]{FlDr88} implies $\|q^{H\uparrow \cB}(f)\|=\|\Ind_H^\cB(S|_{\cB_H})_f\|$ for all $f\in C^*(\cB).$
For the trivial representation $\kappa\colon H\to \bC,$ Fell's absorption principle gives
\begin{align}
\|q^{H\uparrow \cB}(f)\|
& =\| \Ind_H^\cB(S|_{B_H}) \|=\| [S \intform{\otimes}\Ind_H^G(\kappa)]_f \|\leq \| [S\intform{\otimes} \Ind_H^G(U)]_f \| \nonumber \\
& = \|q^{H\uparrow G}_\cB(f)\|;\label{equ:pi is an iso}
\end{align}
which clearly implies $\pi_H^{\cB}$ is isometric.
These arguments can not be extended to non-normal subgroups.
Consider, for example, $\cB$ as the trivial bundle over $G$ with constant fiber $\bC$ and $H\leqslant G$ such that the canonical *-homomorphism $C^*(H)\to M(C^*(G))$ is not  faithful.

How was that Exel and Ng go to define $\pi_{\{e\}}^{\cB}$ and prove it is faithful even for non saturated $\cB$?
The short answer is they developed a version of \eqref{equ:Fell abst prin} where $\cong$ is replaced by a weak equivalence $\approx$  of representations.
\textit{Exel-Ng's absorption principle} states that
for any given non degenerate representation $S$ of $\cB,$ 
\begin{equation}\label{equ:Exel Ng abst prin}
S\otimes \lambda\approx \Ind_{\{e\}}^\cB(S|_{B_e}).
\end{equation}
The statement and proof of this claim is implicit in that of \cite[Theorem 2.14]{ExNg} when the authors show $\lambda_\cB(a)\otimes_{S|_{B_e}} 1=0\Leftrightarrow [S \intform{\otimes} \lambda]_a=0.$
Notice \eqref{equ:Exel Ng abst prin} suffices to define $\pi_{\{e\}}^\cB$ and make \eqref{equ:pi is an iso} work.

All the facts presented before led us to the following questions concerning a general Fell bundle $\cB=\{B_t\}_{t\in G}$ (saturated or not) and $H\leqslant G.$

\begin{enumerate}
	\item\label{item: all rep are B positive} Are all representations of $\cB_H$ inducible to $\cB$ (i.e. $\cB-$positive)?
	
	\item\label{item: q construction of L2H} Can one imitate Exel-Ng's construction of $\lambda_\cB\colon C^*(\cB)\to \bB(L^2_e(\cB))$ using $H$ instead of $\{e\}$?
	More specifically, we ask for the possibility of constructing a right $C^*(\cB_H)-$Hilbert module $L^2_H(\cB)$ and  *-homomorphism 
	\begin{equation*}
	\regrep{H\cB}\colon C^*(\cB)\to \bB(L^2_H(\cB))
	\end{equation*} such that, for every  representation $T$ of $\cB_H,$ one has $\regrep{H\cB}\otimes_{\intform{T}}1 =\intform{\Ind}_{H}^\cB(T)$.
	In case this can be done, it would follow immediately that $C^*_{H\uparrow \cB}(\cB)=\regrep{H\cB}(C^*(\cB)).$
	
	\item\label{item: q general abs prin} Is there any general weak form of \eqref{equ:Fell abst prin}? We mean something similar to \eqref{equ:Exel Ng abst prin} which one may use to define $\pi_H^{\cB}.$
	
	\item\label{item: q universal algebras agree} Assuming questions \eqref{item: q construction of L2H} and \eqref{item: q general abs prin} admit affirmative answers, under which circumstances is $\pi_H^{\cB}$ an isomorphism?
	In other words, when is $\regrep{H\cB}(C^*(\cB))$ universal for the representations of $\cB$ of the form $S\intform{\otimes} \Ind_H^G(U)$?
	We know this is true if $H$ is normal and $\cB$ saturated, is saturation really necessary?
\end{enumerate}

The outline of this article is as follows.
After this introduction, the first main theorem gives an affirmative answer to question \eqref{item: all rep are B positive} and right after that we construct the *-homomorphism $\regrep{H\cB}$ of question \eqref{item: q construction of L2H}.
We then prove $C^*_{H\uparrow\cB}(\cB)$ and $C^*_{H\uparrow G}(\cB)$ have certain universal properties for different families of representations of $\cB.$
Those characterizations are used to compute $C^*_{H\uparrow\cB}(\cB)$ and $C^*_{H\uparrow G}(\cB)$ in a concrete example revealing that many of our theorems fail if some hypotheses are removed.
As an application we construct certain ``exotic coactions'' $\delta\colon C^*_{H\uparrow G}(\cB)\to M(C^*_{H\uparrow G}(\cB)\otimes Q^G_H)$ for specific quotients $Q^G_H$ of $C^*(G).$
In the last part of section \ref{sec:universal} we show that if $G$ is inner amenable and $C^*_\red(\cB)\otimes_{\max} C^*_\red(G)=C^*_\red(\cB)\otimes C^*_\red(G),$ then $C^*(\cB)=C^*_\red(\cB).$

Our answer to question \eqref{item: q general abs prin} occupies most of section \ref{sec: abs prin}, which ends with a series of corollaries.
In one of them we construct a *-homomorphism that, in some situations, happens to be the inverse of the $\pi^\cB_H$ of \eqref{equ: pi uparrow H}.
Another of the corollaries is an extension of Exel-Ng's absorption principles to normal subgroups other than $\{e\}.$

In the fifth and final section we study the dependence of $C^*_{H\uparrow \cB}(\cB)$ and $C^*_{H\uparrow G}(\cB)$ with respect to $H.$
The main result states that if the normalizer of $H$ is open in $G,$ then $C^*(\cB_H)=C^*_\red(\cB_H)$ if and only if $C^*_{H\uparrow \cB}(\cB)= C^*_\red(\cB).$
As a corollary of this we get that $C^*(\cB)=C^*_\red(\cB)$ implies $C^*(\cB_H)=C^*_\red(\cB_H).$

\section{Positivity and Induction}\label{sec:positivity}

We take from \cite{FlDr88} all the definitions, constructions and results concerning representations, C*-algebras, Banach and C*-algebraic bundles (the latter are called Fell bundles).
Any notation different from that of \cite{FlDr88} will be introduced along with its corresponding explanation.

Given a (right or left) Hilbert module $X,$ we denote $\bB(X)$ the C*-algebra of adjointable maps from $X$ to $X.$ 
Whenever $A$ is a C*-algebra and $\pi\colon A\to \bB(X)$ a non degenerate *-homomorphism, $\overline{\pi}\colon M(A)\to \bB(X)$ stands for the unique extension of $\pi.$
We use the symbol $\langle\ ,\ \rangle$ to denote inner products (except when we recall the definition of Banach *-algebraic bundle).
All the groups in this article are assumed to be LCH and integration with respect to left invariant Haar measures will be indicated by $dt.$
The modular function of a group $G$ will be denoted  $\Delta_G.$

Whenever $A$ and $B$ are C*-algebras and there exists a canonical *-homomorphism $\pi\colon A\to B,$ the expression $A=B$ means $\pi$ is an isomorphism and we think $a=\pi(a)$ for all $a\in A.$
This is the case when we write $C^*(G)=C^*_\red(G)$ (i.e. $G$ is amenable) or $A\otimes B=A\otimes_{\max}B.$

When we say $\cB=\{B_t\}_{t\in G}$ is a Banach *-algebraic bundle we are implicitly assuming the existence of a structure $\langle B,\pi,\cdot\rangle $  making $\cB\equiv \langle B,\pi,\cdot\rangle$ a Banach *-algebraic bundle in the sense of \cite[VIII 3.1]{FlDr88}.
When we write $b\in \cB$ we mean $b\in B.$
Notice that the fibre $B_e$ over the unit $e\in G$ is a Banach *-algebra.
By definition, $\cB=\{B_t\}_{t\in G}$ is a Fell bundle (i.e. C*-algebraic bundle) if $\|b^*b\|=\|b\|^2$ and $b^*b$ is positive in the C*-algebra $B_e,$ for all $b\in B.$
Given a Banach *-algebraic (Fell) bundle $\cB=\{B_t\}_{t\in G}$ and $H\leqslant G,$ the reduction $\cB_H:=\{B_t\}_{t\in H}$ is a Banach *-algebraic (Fell, respectively) bundle with the structure inherited from $\cB.$

The concrete and abstract induction processes of \cite{FlDr88} can be applied to any $\cB-$positive representation $S\colon \cB_H\to \bB(X)$ and give (two) unitary equivalent representations of $\cB$ \cite[XI 9.26]{FlDr88}, any of which we denote $\Ind_{H}^\cB(S)$ and call \textit{the} representation of $\cB$ induced by $S$.
The definition of $\cB-$positivity we adopt is that of \cite[XI 8.6]{FlDr88}.
By \cite[XI 8.9]{FlDr88}, a representation $S\colon \cB_H\to \bB(X)$ is $\cB-$positive if for every coset $rH \in G/H\equiv \{tH\colon t\in G\},$ every  integer $n>0,$ all $b_1,\ldots,b_n\in \cB_{rH},$ and all $\xi_1,\ldots,\xi_n\in X,$
\begin{equation}\label{equ:positivity}
\sum_{i,j=1}^n \langle S_{b_j^*b_i}\xi_i,\xi_j\rangle \geq 0.
\end{equation}
We give an alternative formulation in corollary \ref{cor: weak form of characterization of positivity}.

In \cite[pp 1159]{FlDr88} Fell proves the theorem below for saturated bundles and asks if saturation can be removed from the hypotheses.
It indeed can and, as pointed out by Fell in the same page, this leads to a simpler formulation of positivity (corollary \ref{cor: weak form of characterization of positivity}).

\begin{theorem}\label{thm:positivity}
	If $\cB=\{B_t\}_{t\in G}$ is a Fell bundle and $H\leqslant G,$ then all the representations of $\cB_H$ are $\cB-$positive.
\end{theorem}
\begin{proof}
	Let $T\colon \cB_H\to \bB(X)$ be a representation.
	Fix $t\in G$ and elements $b_1,\ldots,b_n\in \cB_{tH}.$
	Take $s_1,\ldots,s_n\in H$ such that $b_j\in B_{ts_j}$ ($j=1,\ldots, n$).
	Set $\mathfrak{s}:=(s_1,\ldots, s_n)$ and $t\mathfrak{s}:=(ts_1,\ldots,ts_n)$ and define the matrix space
	\begin{equation*}
	\mathbb{M}_{t\mathfrak{s}}(\cB):=\{ (M_{i,j})_{i,j=1}^n\colon M_{i,j}\in B_{(ts_i)^{-1}(ts_j)},\ \forall\ i,j=1,\ldots,n \}.
	\end{equation*}
	It is of key importance to notice that  $\mathbb{M}_{t\mathfrak{s}}(\cB)\equiv \mathbb{M}_{\mathfrak{s}}(\cB_H).$
	
	By \cite[Lemma 2.8]{AbFrrEquivalence}, $\mathbb{M}_{t\mathfrak{s}}(\cB)$ is a C*-algebra with usual matrix multiplication as product and $*-$transpose as involution.
	A quick way of proving this is by taking a representation $R\colon \cB\to \bB(Y)$ with all the restrictions $R|_{B_t}$ being isometric (which exists by \cite[VIII 16.10]{FlDr88}) and to identify $\mathbb{M}_{t\mathfrak{s}}(\cB)$ with the concrete C*-algebra
	\begin{equation*}
	\{ (R_{M_{i,j}})_{i,j=1}^n\colon M_{i,j}\in B_{s_i^{-1}s_j},\ \forall\ i,j=1,\ldots,n\}\subset \bB(Y^n).
	\end{equation*}
	
	The matrix $M:=(b_i^*b_j)_{i,j=1}^n$ belongs to $\mathbb{M}_{t\mathfrak{s}}(\cB)$ and, regarding $\cB$ as a $\cB-\cB-$equivalence bundle, we can easily adapt the proof of \cite[Lemma 2.8]{AbFrrEquivalence} to show that $M$ is positive in $\mathbb{M}_{t\mathfrak{s}}(\cB)\equiv \mathbb{M}_{\mathfrak{s}}(\cB_H).$
	An alternative (direct) proof is as follows.
	
	Notice that for all $M=(M_{i,j})_{i,j=1}^n\in \mathbb{M}_{t\mathfrak{s}}(\cB)$ and all $i=1,\ldots,n$ we have $M_{i,i}\in B_e.$
	So we may define a \textit{trace function} $\trace\colon \mathbb{M}_{t\mathfrak{s}}(\cB)\to B_e$ by $\trace(M):=\sum_{i=1}^n M_{i,i}.$
	The operations 
	\begin{align*}
	\mathbb{M}_{t\mathfrak{s}}(\cB)\times & B_e\to B_e  & (M,b)&\mapsto (M_{i,j}b)_{i,j=1}^n\\
	\langle\ ,\ \rangle_{B_e}\colon \mathbb{M}_{t\mathfrak{s}}(\cB)\times & \mathbb{M}_{t\mathfrak{s}}(\cB)\to B_e & (M,N)&\mapsto \langle M,N\rangle_{B_e}:=\trace(M^*N)
	\end{align*}
	make of $\mathbb{M}_{t\mathfrak{s}}(\cB)$ a full right $B_e-$Hilbert module which we denote $\mathbb{M}_{t\mathfrak{s}}(\cB)_{B_e}.$
	
	For every $M\in \mathbb{M}_{t\mathfrak{s}}(\cB)$ the operator $\phi_M\colon \mathbb{M}_{t\mathfrak{s}}(\cB)_{B_e}\to \mathbb{M}_{t\mathfrak{s}}(\cB)_{B_e}, N\mapsto MN,$ is adjointable because for all $N,P\in \mathbb{M}_{t\mathfrak{s}}(\cB)_{B_e},$ 
	\begin{equation*}
	\langle \phi_M N,P\rangle_{B_e}=
	\trace((MN)^*P)=\trace(N^*(M^*P))=\langle N,\phi_{M^*}P\rangle_{B_e}.
	\end{equation*}
	Moreover, the natural representation $\phi\colon \mathbb{M}_{t\mathfrak{s}}(\cB)\to \bB(\mathbb{M}_{t\mathfrak{s}}(\cB)_{B_e}),$ $M\mapsto \phi_M,$ is a *-homomorphism.
	
	To prove that $\phi$ is faithful we take an approximate unit $\{e_i\}_{i\in I}$ of $\mathbb{M}_{t\mathfrak{s}}(\cB).$
	For all $M\in \mathbb{M}_{t\mathfrak{s}}(\cB)$ we have
	$\sum_{i,j=1}^n (M_{i,j})^*M_{i,j} = \trace(M^*M)=\lim_i\trace((Me_i)^*M)=\lim_i \langle \phi_Me_i,M\rangle_{B_e}.$
	Thus $\phi_M=0$ implies $\sum_{i,j=1}^n (M_{i,j})^*M_{i,j}=0$ and this yields $M=0.$
	
	Now that we know $\phi$ is faithful, to prove that  $M:=(b_i^*b_j)_{i,j=1}^n \geq 0 $ (in $\mathbb{M}_{t\mathfrak{s}}(\cB)\equiv \mathbb{M}_{\mathfrak{s}}(\cB_H)$) it suffices to show that $\phi_M\geq 0.$
	This is the case because for all $N\in \mathbb{M}_{t\mathfrak{s}}(\cB)_{B_e}$ we have
	\begin{equation*}
	\langle \phi_M N,N\rangle_{B_e} =
	\trace((\phi_M N)^*N)=\trace(N^*M^*N)=\sum_{i=1}^n\left(\sum_{j=1}^n b_jN_{j,i}\right)^*\left(\sum_{k=1}^n b_kN_{k,i}\right)\geq 0;
	\end{equation*}
	where the last inequality follows from the fact that the definition of Fell bundle requires $c^*c$  to be positive in $B_e$ for all $c\in B.$
	
	If $N:= M^{1/2}\in\mathbb{M}_\mathfrak{ts}(\cB)\equiv  \mathbb{M}_\mathfrak{s}(\cB_H),$ then all the entries $N_{i,j}$ of $N$ belong to $\cB_H$ and for all $\xi_1,\ldots, \xi_n\in X$ we have
	\begin{equation*}
	\sum_{i,j=1}^n\langle T_{b_j^*b_i}\xi_i,\xi_j\rangle 
	= \sum_{i,j,k=1}^n \langle T_{{N_{k,j}}^*N_{k,i}} \xi_i,\xi_j\rangle
	= \sum_{k=1}^n \langle \sum_{i=1}^n T_{N_{k,i}}\xi_i,\sum_{j=1}^nT_{N_{k,j}}\xi_j\rangle\geq 0;
	\end{equation*}
	proving that $T$ is $\cB-$positive. 
\end{proof}

\begin{corollary}[c.f. {\cite[XI 11.11]{FlDr88}}]\label{cor: weak form of characterization of positivity}
	Let $\cB=\{B_t\}_{t\in G}$ be a Banach *-algebraic bundle and $H$ a closed subgroup of $G.$
	For any representation $S\colon \cB_H\to \bB(X)$ the following  conditions are equivalent:
	\begin{enumerate}
		\item $S$ is $\cB-$positive.
		\item $\langle S_{b^*b}\xi,\xi\rangle \geq 0$ for all $b\in \cB$ and $\xi\in X.$
	\end{enumerate}
\end{corollary}

\subsection{Fell's abstract induction process for Fell fundles}

Take a Fell bundle $\cB=\{B_t\}_{t\in G}.$
We denote $C_c(\cB),$ and not $\mathcal{L}(\cB)$ as in \cite{FlDr88}, the set of continuous cross-sections of $\cB$ with compact support.
This set is in fact a dense *-subalgebra of the $L^1-$cross-sectional algebra $L^1(\cB)$ of \cite[VIII 5.6]{FlDr88}.
The cross-sectional C*-algebra of $\cB,$ $C^*(\cB),$ is the enveloping C*-algebra of $L^1(\cB)$ and we know from \cite[VIII 16.4]{FlDr88}  that $L^1(\cB)$ is reduced, meaning that we may think of $L^1(\cB)$ as a dense *-subalgebra of $C^*(\cB).$
The integrated form of a representation $T\colon \cB\to \bB(X)$ is the unique representation $\intform{T}\colon C^*(\cB)\to \bB(X)$ such that $\intform{T}_f \xi = \int_G T_{f(t)}\xi\, dt$ for all $f\in C_c(\cB)$ and $\xi\in X.$
All representations of $C^*(\cB)$ arise this way\footnote{So they can be ``disintegrated''.} and $\intform{T}$ determines  $T.$

The definition of weak containment (and equivalence) of representations we adopt are those of \cite[VII 1]{FlDr88} and \cite[VIII 21]{FlDr88}.
Given two sets of representations, $\mathcal{S}$ and  $\mathcal{T},$ the expressions $\mathcal{S}\preceq \mathcal{T}$ and $\mathcal{S}\approx \mathcal{T}$ mean that $\mathcal{S}$ is weakly contained in $\mathcal{T}$ and that $\mathcal{S}$ is weakly equivalent to $\mathcal{T},$ respectively.

The basic ingredients we need to perform Fell's abstract induction process are a closed subgroup $H$ of $G;$ a representation $T\colon \cB_H\to \bB(X);$ the ``normalized restriction''
\begin{align}
p \colon & C_c(\cB)\to C_c(\cB_H)  &  p(f)(t)&=\left(\frac{\Delta_G(t)}{\Delta_H(t)}\right)^{1/2}f(t) \label{equ: p}
\end{align}
and the action $C_c(\cB)\times C_c(\cB_H) \to C_c(\cB),$  $(f,u)\mapsto fu,$ given by
\begin{equation}
\qquad  fu(t)= \int_H f(ts)u(s^{-1})\left(\frac{\Delta_G(s)}{\Delta_H(s)}\right)^{1/2}\, ds.\label{equ: action to induce}
\end{equation}

We abbreviate $\frac{\Delta_G(t)}{\Delta_H(t)}$ to $\Delta_H^G(t)$ and write $p_H$ instead of $p$ only when not doing so may cause any confusion.

Any representation $T\colon \cB_H\to \bB(X)$ is $\cB-$positive and this implies there is a unique pre-inner product $[\ ,\ ]_T$ on the algebraic tensor product $C_c(\cB)\odot X$ such that, for all $f,g\in C_c(\cB)$ and $\xi,\eta\in X,$ $[ f\odot \xi,g\odot \eta]_T = \langle \intform{T}_{p(g^**f)}\xi,\eta\rangle.$ 
We denote $X^p_T$ the Hilbert space obtained as the completion of the quotient
\begin{equation*}
C_c(\cB)\odot X /\{u\in C_c(\cB)\odot X \colon [ u,u]_T=0\}
\end{equation*}
with respect to the natural (quotient) inner product, which we denote $\langle \ , \ \rangle.$

The image of an elementary tensor $f\odot \xi\in C_c(\cB)\odot X$ (via the quotient map and the inclusion into the completion) will be denoted $f\otimes_T \xi.$
By construction, 
\begin{equation}\label{equ:inner product on induced space}
\langle f\otimes_T\xi,g\otimes_T\eta\rangle = \langle \intform{T}_{p(g^**f)}\xi,\eta\rangle =  \langle \xi,\intform{T}_{p(f^**g)} \eta\rangle.
\end{equation}

By \cite[XI 9.26]{FlDr88} the (abstractly) induced representation  $\Ind^\cB_H(T)\colon \cB\to \bB(X^p_T),$ and its integrated form $\intform{\Ind}_H^\cB(T) \colon C^*(\cB)\to \bB(X^p_T) $
can be characterized by saying that 
\begin{align*}
\Ind_H^\cB(T)_b(f\otimes_T\xi) &=(bf)\otimes_T\xi  & \intform{\Ind}_H^\cB(T)_g (f\otimes_T\xi)=(g*f)\otimes_T\xi
\end{align*}
for all $b\in \cB,$ $f,g\in C_c(\cB)$ and $\xi\in X,$ where the action of $\cB$ on $C_c(\cB)$ used in the first identity is 
\begin{align*}
\cB\times C_c(\cB)&\to C_c(\cB)  &  (b\in B_s,f)&\mapsto bf  & (bf)(t)&=bf(\smu t).
\end{align*}

\begin{remark}\label{rem:trivial bundle}
	The trivial one dimensional complex Banach bundle over $G,$ $\cT_G=\{\bC\delta_t\}_{t\in G},$ is a saturated Fell bundle with the operations $(z\delta_r)(w\delta_s)=zw\delta_{rs}$ and $(z\delta_r)^* =\overline{z}\delta_r.$
	We can easily identify $C_c(\cT_G),$ $L^1(\cT_G)$ and $C^*(\cT_G)$ with $C_c(G),L^1(G)$ and $C^*(G),$ respectively.
	The unitary representations of $G$ are in one to one correspondence with the non degenerate representations of $\cT_G.$
	Up to this identification, the induction process $U\mapsto \Ind_H^G(U)$ is the same as $U\mapsto \Ind_{H}^{\cT_G}(U).$
\end{remark}

\subsubsection{The induction module}
With $\langle\ ,\ \rangle_{H}^\cB\colon  C_c(\cB)\times C_c(\cB)\to C_c(\cB_H)$ defined by $\langle f,g\rangle_H^\cB:=p(f^**g),$ $C_c(\cB)$ becomes a right $C_c(\cB_H)-$rigged left $C_c(\cB)-$module (in the sense of \cite{FlDr88}) with the action \eqref{equ: action to induce} on the right and the natural action by convolution on the left.
If we consider $C_c(\cB_H)$ as a dense *-subalgebra of $C^*(\cB_H),$ then $\langle\ ,\ \rangle_{H}^\cB$ is positive in the sense that $\langle f ,f \rangle_{H}^\cB\geq 0$ for all $f\in C_c(\cB).$
Indeed, take a non degenerate representation $T\colon \cB_H\to \bB(X)$ with faithful integrated form.
Then $\intform{T}_{\langle f ,f \rangle_{H}^\cB}\equiv \intform{T}_{p(f^**f)}\geq 0$ because $T$ is $\cB-$positive \cite[XI 8.6--8.9]{FlDr88}.

We are now in the situation of \cite[Lemma 2.16]{Raeburn1998morita}, so there exists a right $C^*(\cB_H)-$Hilbert module $L^2_H(\cB)$ (with inner product $\langle\ ,\ \rangle_{C^*(\cB_H)}$) and a linear map $q\colon C_c(\cB)\to L^2_H(\cB)$ with dense image and $\langle q(f),q(g)\rangle_{C^*(\cB_H)}=p(f^**g)$ for all $f,g\in C_c(\cB).$
To prove $q$ is faithful we start by noticing $q(f)=0$ implies $p(f^**f)(e)=0,$ which translates to $\int_G f(t)^*f(t)\, dt=0.$
This last condition implies $f=0$ because $t\mapsto f(t)^*f(t)$ is a continuous function with compact support from $G$ to the positive cone of $B_e.$

Now that we know $q$ is injective, we delete any reference to it and think of $C_c(\cB)$ as a dense subspace of $L^2_H(\cB).$
For example, we make no distinction between $\langle f,g\rangle_{C^*(\cB_H)},$ $\langle f,g\rangle_H^\cB$ and $p(f^**g).$

The following remark will be used repeatedly (and even without mention) in the rest of the article.

\begin{remark}\label{rem:construction of isometry}
	Given a non empty set $A,$ Hilbert spaces $X$ and $Y$ and functions $x\colon A\to X$ and $y\colon A\to Y$ such that  $\langle x(a),x(b)\rangle = \langle y(a),y(b)\rangle$ for all $a,b\in A,$ there exists a unique linear isometry $I\colon \cspn\{x(a)\colon a\in A\}\to \cspn\{y(a)\colon a\in A\}$ such that $I\circ x=y.$
\end{remark}

The proposition we are about to state is a reinterpretation of \cite[XI 9.26]{FlDr88}.
It reveals that when working with Fell bundles, the induction process can be carried out using Hilbert modules and not just left rigged modules.

\begin{proposition}\label{prop:construction of general reg rep}
	There exists a unique *-homomorphism $\regrep{H\cB}\colon C^*(\cB)\to \bB(L^2_H(\cB)) $ such that $\regrep{H\cB}_f g=f*g$ for all $f,g\in C_c(\cB).$
	Moreover, for any representation $T\colon \cB_H\to \bB(X)$ it follows that 
	\begin{equation*}
	\regrep{H\cB}\otimes_{\intform{T}}1\cong \intform{\Ind}_H^\cB(T).
	\end{equation*}
\end{proposition}
\begin{proof}
	Take any representation $T\colon \cB_H\to \bB(X).$
	By \cite[XI 9.26]{FlDr88} the representation $\intform{T}|_{C_c(\cB)}$ is inducible to $C_c(\cB)$ via the conditional expectation $p$ of \eqref{equ: p} and the resulting induced representation is $\intform{\Ind}_H^\cB(T)|_{C_c(\cB)}.$
	Then, for all $f,g\in C_c(\cB)$ and $\xi\in X$ we have
	\begin{equation*}
	\langle \intform{T}_{\langle f*g,f*g\rangle_H^\cB}\xi,\xi\rangle 
	=\|\Ind_H^\cB(T)_f(g\otimes_T\xi)\|^2
	\leq \|f\|_1^2\|g\otimes_T\xi\|^2=\|f\|_1^2\langle \intform{T}_{\langle g,g\rangle_H^T}\xi,\xi\rangle.
	\end{equation*}
	Since we can choose $T$ so that $\intform{T}$ is faithful, we have $\langle f*g,f*g\rangle_H^\cB\leq \|f\|_1^2\langle g,g\rangle_H^\cB$ and
	\begin{equation*}
	\langle f*g,h\rangle_H^\cB = p((f*g)^**h)=p(g^** (f^**g))=\langle g,f^**h\rangle_H^\cB
	\end{equation*}
	for all $f,g,h\in C_c(\cB).$
	This implies the existence of a map $\regrep{0}\colon C_c(\cB)\to \bB(L^2_H(\cB))$ such that $\regrep{0}_fg=f*g,$ $\|\regrep{0}_f\|\leq\|f\|_1$ and $(\regrep{0}_f)^*=\regrep{0}_{f^*}.$
	
	Note that $\regrep{0}$ is a homomorphism of *-algebras which is contractive with respect to $\|\ \|_1,$ so it admits a unique extension to a *-homomorphism $\regrep{1}\colon L^1(\cB)\to \bB(L^2_H(\cB));$ which we can extend in a unique way to a *-homomorphism $\regrep{H\cB}\colon C^*(\cB)\to \bB(L^2_H(\cB)).$
	
	The tensor product $L^2_H(\cB)\otimes_{\intform{T}}X$ is the closed linear span of elementary tensors $f\otimes_{\intform{T}}\xi$ ($f\in C_c(\cB)$ and $\xi\in X$) with $\langle f\otimes_{\intform{T}}\xi,g\otimes_{\intform{T}}\eta\rangle =\langle \intform{T}_{\langle g,f\rangle_H^\cB}\xi,\eta\rangle = \langle f\otimes_T\xi,g\otimes_T\eta\rangle.$
	So there exists a unique unitary $U\colon L^2_H(\cB)\otimes_{\intform{T}}X\to X^p_T$ sending $f\otimes_{\intform{T}}\xi$ to $f\otimes_T\xi.$
	This operator intertwines $\regrep{H\cB}\otimes_{\intform{T}}1$ and $\intform{\Ind}_H^\cB(T)$ because for all $f,g\in C_c(\cB)$ and $\xi\in X$ we have
	\begin{equation*}
	U^* \intform{\Ind}_H^\cB(T)_fU (g\otimes_{\intform{T}}\eta)=U^*((f*g)\otimes_T\eta)=(f*g)\otimes_{\intform{T}}\eta= [\regrep{H\cB}\otimes_{\intform{T}}1]_f(g\otimes_{\intform{T}}\eta);
	\end{equation*}
	which implies $U^* \intform{\Ind}_H^\cB(T)_fU=[\regrep{H\cB}\otimes_{\intform{T}}1]_f$ for all $f\in C^*(\cB).$
\end{proof}

\begin{remark}
	One may use \cite[VIII 12.7]{FlDr88} to disintegrate $\regrep{H\cB}$ into a Fr\'{e}chet representation ${\regrep{H\cB}}'$ which happens to be given by ${\regrep{H\cB}}'_b f = bf,$ for all $b\in \cB$ and $f\in C_c(\cB)\subset L^2_H(\cB).$
	This ultimately follows from the fact that for all $b\in \cB$ and $f,g\in C_c(\cB),$   ${\regrep{H\cB}}'_b (f*g) ={\regrep{H\cB}}'_b \regrep{H\cB}_f g = \regrep{H\cB}_{bf}g  = (bf)*g = b(f*g).$
\end{remark}

\begin{remark}[Systems of Imprimitivity]\label{rem:systems of imprimitivity}
	As pointed out in \cite[XI 14.4]{FlDr88}, there is a natural action $C_0(G/H)\times C_c(\cB)\to C_c(\cB),$ $(f,g)\mapsto fg$ where $fg (r) = f(rH)g(r).$
	Moreover, Fell shows that given a representation $T\colon \cB_H\to \bB(X),$ the action of $C_0(G/H)$ induces a representation
	\begin{align*}
	\psi^T & \colon C_0(G/H)\to \bB(L^2_H(\cB)\otimes_{\intform{T}}X)  & \psi^T_f(g\otimes_T \xi)&=fg\otimes_T\xi.
	\end{align*}
	In particular, $\langle \xi,\intform{T}_{\langle fg,fg\rangle}\xi\rangle = \| \psi^T_f(g\otimes_T\xi) \|^2\leq \|f\|^2\|g\otimes_T\xi\|^2= \|f\|^2\langle \xi,\langle g,g\rangle\xi\rangle$ and this yields  $\langle fg,fg\rangle\leq \|f\|^2\langle g,g\rangle.$
	This is the key fact one needs to prove the existence of a unique *-homomorphism $\psi\colon C_0(G/H)\to \bB(L^2_H(\cB))$ such that $\psi_fg = fg$ for all $(f,g)\in C_0(G/H)\times C_c(\cB).$
	It turns out that $\psi$ is non degenerate and $\psi^T_f = \psi_f\otimes_{\intform{T}}1.$
	The pair $(\regrep{H\cB},\psi)$ is a kind of universal system of imprimitivity because, according to \cite[XI 14.3--14.4]{FlDr88}, by disintegrating  $\regrep{H\cB}\otimes_{\intform{T}}1$ and $\psi\otimes_{\intform{T}}1$ one gets $\Ind_H^\cB(T)$ and the projection-valued measure induced by $T.$
	It then should come as no surprise that, denoting $\tau$ the natural action of $G$ on $C_0(G/H),$ one obtains ${\regrep{H\cB}}'_b \psi_f = \psi_{\tau_t(f)} {\regrep{H\cB}}'_b$ for all $b\in B_t$ and $f\in C_0(G/H).$
\end{remark}

\begin{remark}\label{rem:H=e}
	If $H=\{e\},$ then $L^2_H(\cB)=L^2_e(\cB)$ and $\regrep{H\cB}=\lambda_\cB.$
	Indeed, in this situation $B_e\equiv C_c(\cB_H)$ and the inner product of $f,g\in C_c(\cB)$ in $L^2_e(\cB)$ is $\int_Gf(t)^*g(t)\, dt=f^**g(e)\equiv p(f^**g)=\langle f,g\rangle_H^\cB.$
	The action of \eqref{equ: action to induce} reduces to $(fb)(t) =f(t)b$ and this is the action used by Exel and Ng to construct $L^2_e(\cB).$
	Hence, $L^2_H(\cB)=L^2_e(\cB).$
	The rest is an immediate consequence of the las proposition above and \cite[Proposition 2.6]{ExNg}.
\end{remark}

\begin{remark}\label{rem:H=G}
	If $H=G$ then $L^2_H(\cB)$ is the C*-algebra $C^*(\cB)$ considered as a right $C^*(\cB)-$Hilbert module and $\regrep{G\cB}\colon C^*(\cB)\to \bB(C^*(G))$ is the natural inclusion.
	This is the case because $\langle f,g\rangle_G^\cB = p(f^**g)=f^**g$ and the action \eqref{equ: action to induce} is given by the convolution product.
\end{remark}

\section{Universal properties of cross-sectional C*-algebras}\label{sec:universal}

Let $\cB$ be a Fell bundle and $\cF$ a non empty family of representations of $\cB.$
We say a C*-algebra $A$ is universal for $\cF$ if there exists a surjective *-homomorphism $\pi\colon C^*(\cB)\to A$ such that (i) for every $(S\colon \cB\to \bB(X))\in \cF,$ there exists a representation $\rho^S\colon A\to \bB(X)$ with $\rho^S\circ \pi=\intform{S}$ and (ii) for some $S\in \cF,$  $\rho^S$ is faithful. 
Notice $\rho^S$ is unique because $\pi$ is surjective.

Say we have another surjective *-homomorphism $\pi'\colon C^*(\cB)\to A'$ satisfying (i) and (ii).
Let $S\colon \cB\to \bB(X)$ and $T\colon \cB\to \bB(Y)$ be members of $\cF$ such that the respective representations $\rho^S\colon A\to \bB(X)$ and ${\rho'}^T\colon A'\to \bB(Y)$ are faithful.
By uniqueness,  both $\rho^{S\oplus T} = \rho^S\oplus \rho^T$ and ${\rho'}^{S\oplus T}$ are faithful and we get that for all $f\in C^*(\cB),$ $\|\pi(f)\| = \|[T\intform{\oplus }S]_f\|=\|\pi'(f)\|.$
This clearly implies the existence of a unique isomorphism $\mu\colon A\to A'$ such that $\mu\circ \pi = \pi'.$

The next propositions imply $C^*_{H\uparrow\cB}(\cB)$ is universal for the representations of $\cB$ induced from representations of $\cB_H;$ while $C^*_{H\uparrow\cB}(\cB)$ is for those of the form $T\otimes \Ind_H^G(U),$ $T$ being a representation of $\cB$ and $U$ one of $H.$

\begin{proposition}\label{prop: CHBB universal property}
	Let $\cB=\{B_t\}_{t\in G}$ be a Fell bundle, $H\leqslant G$ and $T\colon \cB_H\to \bB(X)$ a representation.
	Then there (exists) a unique representation $\rho\colon C^*_{H\uparrow \cB}(\cB)\to \bB(X^p_T)$ such that $\rho\circ q^{H\uparrow\cB}=\intform{\Ind}_H^\cB(T).$
	Moreover, if $\intform{T}$ is faithful then so it is $\rho.$
\end{proposition}
\begin{proof}
	Let $J\subset C^*(\cB)$ be the intersection of all the kernels of integrated forms of representations of $\cB$ induced from representations of $\cB_H.$
	We obviously have $J\subset \ker(\intform{\Ind}_H^\cB(T)),$ so there exists a unique function $\rho\colon C^*(\cB)/J\to \bB(X^p_T)$ such that $\rho\circ q^{H\uparrow \cB} = \intform{\Ind}_H^\cB(T)$ (implying $\rho$ is a representation).
	
	Assume $\intform{T}$ is faithful and let $S$ be any other representation of $\cB_H.$
	Both $S$ and its non degenerate part induce the same representation of $\cB,$ so we may assume both $S$ and $T$ are non degenerate.
	We have $\intform{S}\preceq \intform{T}$ because $\intform{T}$ is faithful.
	By the definition of weak containment for bundles, $S\preceq T$ and the continuity of the induction process with respect the regional topology \cite[XI 12.4]{FlDr88} implies $\intform{\Ind}_H^\cB(S)\preceq \intform{\Ind}_H^\cB(T).$
	Hence, $\ker(\intform{\Ind}_H^\cB(T))\subset \ker (\intform{\Ind}_H^\cB(S))$ and it follows that $J=\ker(\intform{\Ind}_H^\cB(T))$ or, in other words, that $\rho$ is faithful.
\end{proof}

\begin{proposition}\label{prop: CHGB universal property}
	Let $\cB=\{B_t\}_{t\in G}$ be a Fell bundle and $H\leqslant G.$
	Given non degenerate representations $T\colon \cB\to \bB(X)$ and $U\colon H\to \bB(Y)$ there exists a (unique) representation $\rho\colon C^*_{H\uparrow G}(\cB)\to \bB(X\otimes Y)$ such that $\rho\circ q^{H\uparrow G}_\cB =T\intform{\otimes} \Ind_H^G(U).$
	Moreover, $\rho$ is faithful if $\intform{T}$ and $\intform{U}$ are.
\end{proposition}
\begin{proof}
	By definition, we may think of $C^*_{H\uparrow G}(\cB)$ as the image of the *-homomorphism $\delta\colon C^*(\cB)\to M(C^*(\cB)\otimes Q)$ associated to the quotient map $q\colon C^*(G)\to Q:=C^*(G)/I;$ with $I\subset C^*(G)$ the intersection of all the kernels of integrated form of *-representations induced from $H.$
	Hence, $\intform{\Ind}_H^G(U)$ factors through a representation $\pi\colon Q\to \bB(Y)$ and we get a representation $\intform{T}\otimes\pi\colon C^*(\cB)\otimes Q\to \bB(X\otimes Y)$ we may extend to the multiplier algebra to get a representation $\overline{\intform{T}\otimes\pi}.$
	
	When we described $\delta\colon  C^*(\cB)\to M(C^*(\cB)\otimes Q)$ in the introduction we made use of canonical unitary representation $\iota\colon G\to M(Q)$ and the action of $\cB$ on $C_c(\cB)\subset C^*(\cB).$
	In fact, $\delta_f (g\otimes z) =\int_G f(t)g\otimes \iota(t)z\, dt$ for all $f,g\in C_c(\cB)$ and $z\in Q.$
	This yields $ \overline{\intform{T}\otimes\pi}\circ \delta = T\intform{\otimes} \Ind_H^G(U).$
	Since we are identifying $C^*_{H\uparrow G}(\cB)=\delta(C^*(\cB)),$ the quotient  $q^{H\uparrow G}_\cB\colon C^*(\cB)\to C^*_{H\uparrow G}(\cB)$ becomes $f\mapsto \delta_f$ and we may therefore define $\rho$ as the restriction of $\overline{\intform{T}\otimes\pi}$ to $C^*_{H\uparrow G}(\cB).$
	Say we have some other representation $\rho'\colon C^*_{H\uparrow G}(\cB)\to \bB(X\otimes Y)$ such that $\rho'\circ q^{H\uparrow G}_\cB=T\intform{\otimes} \Ind_H^G(U).$
	Then $\rho'\circ q^{H\uparrow G}_\cB  =\rho\circ q^{H\uparrow G}_\cB$ and this yields $\rho=\rho'$ because $q^{H\uparrow G}_\cB$ is surjective.
	
	Assume the integrated forms of $T$ and $U$ are faithful.
	Then $\pi$ is faithful because given any other representation $V$ of $H$ we have $\intform{V}\preceq \intform{U}\Rightarrow V\preceq U\Rightarrow \Ind_H^G(V)\preceq \Ind_H^G(U)\Rightarrow \intform{\Ind}_H^G(V)\preceq \intform{\Ind}_H^G(U)$ and this implies $I$ is the kernel of $\intform{\Ind}_H^G(U)$ or, in other words, $\pi$ is faithful and so must be $\overline{\intform{T}\otimes\pi};$ which clearly implies $\rho$ is faithful.
\end{proof}

\begin{remark}\label{rem: CHGB constant by conjutation}
	Say $\cB=\{B_t\}_{t\in G}$ is a Fell bundle and we have conjugated subgroups $H\leqslant G$ and $K=rH\rmu.$
	By \cite[XI 12.21]{FlDr88}, up to unitary equivalences, the classes of representations of $G$ induced from $H$ and $K$ agree.
	Then there exists an isomorphism $\chi\colon C^*_{H\uparrow G}(\cB)\to C^*_{K\uparrow G}(\cB)$ such that $\chi\circ q^{H\uparrow G}_\cB=q^{K\uparrow G}_\cB.$
\end{remark}

In the introduction we showed $C^*_{\{e\}\uparrow \cB}(\cB)=\regrepEN{\cB}(C^*(\cB))$ and we know from remark \ref{rem:H=e} that $\regrepEN{\cB}=\regrep{\{e\}\cB}.$
Thus $C^*_{\{e\}\uparrow \cB}(\cB)= \regrep{\{e\}\cB}(C^*(\cB)).$
This is a particular more general fact.

\begin{corollary}
	For every Fell bundle $\cB=\{B_t\}_{t\in G}$ and $H\leqslant G,$ the quotient map $C^*(\cB)\to \regrep{H\cB}(C^*(\cB)),$ $f\mapsto \regrep{H\cB}(f),$ makes of $\regrep{H\cB}(C^*(\cB))$ a universal C*-algebra for the representations of $\cB$ induced from $\cB_H.$
\end{corollary}
\begin{proof}
	Let $T\colon \cB_H\to \bB(X)$ be a representation.
	By proposition \ref{prop:construction of general reg rep}, there exists a *-homomorphism $\theta\colon \bB(L^2_H(\cB))\to \bB(L^2_H(\cB)\otimes_{\intform{T}}X)$ such that $\theta(M)=M\otimes_{\intform{T}}1$ and $\theta\circ \regrep{H\cB} = \intform{\Ind}_H^\cB(T).$
	The restriction $\theta|_{\regrep{H\cB}(C^*(\cB))}\colon \regrep{H\cB}(C^*(\cB))\to \bB(X^p_T)$ is then the unique representation $\rho$ of $\regrep{H\cB}(C^*(\cB))$ such that $\rho\circ \regrep{H\cB}= \intform{\Ind}_H^\cB(T).$
	If $\intform{T}$ is faithful then $\theta$ is also, which clearly implies $\rho$ is faithful.
\end{proof}

\subsection{An example}\label{ssec:example}

Here we present an explicit example of a Fell bundle $\cB=\{B_t\}_{t\in G}$ and $H\leqslant G$ such that
\begin{itemize}
	\item The normalizer of $H$ is open (because $G$ is discrete).
	\item $\regrepEN{\cB}\colon C^*(\cB)\to C^*_\red(\cB)$ is not an isomorphism.
	\item For all $r\in H,$ $C^*_{rH\rmu \uparrow G}(\cB)=C^*(\cB)$ (see remark \ref{rem: CHGB constant by conjutation}).
	\item $C^*_{H\uparrow\cB}(\cB)=C^*(\cB)$ and if $r\in G\setminus H,$ then $C^*_{rH\rmu \uparrow\cB}(\cB)=C^*_\red(\cB).$
\end{itemize}

The lemma below is of key importance in the rest of the article, in particular in the example presented right after its proof.
We state it in full generality and not only for discrete groups.
By considering trivial Fell bundles over groups with constant fiber $\bC,$ it is easy to use the lemma to prove the analogous statement for unitary representations of groups.

\begin{lemma}\label{lem:wc induction and restriction}
	If $\cB=\{B_t\}_{t\in G}$ is a Fell bundle and the normalizer of $H\leqslant G$ is open, then
	\begin{enumerate}
		\item For every non degenerate representation $T$ of $\cB_H$ we have $T\preceq \Ind_H^\cB(T)|_{\cB_H}.$
		\item Given a non degenerate representation $T\colon \cB\to \bB(Y),$ $\intform{T|_{\cB_H}}$ is faithful if $\intform{T}$ is.
	\end{enumerate}	
\end{lemma}
\begin{proof}
	By \cite[XI 12.8]{FlDr88} and \cite[XI 14.21]{FlDr88}, the first claim holds if $H$ is either open or closed in $H.$
	To deal with the general case we let $N$ be the normalizer of $H$ in $G.$
	Using induction in stages \cite[XI 12.15]{FlDr88}, the continuity of the induction and restriction processes with respect to the regional topology \cite[XI 12.4 \& VIII 21.20]{FlDr88} we get
	\begin{equation*}
	\Ind_H^\cB(T)|_{\cB_H}\cong \Ind_N^\cB(\Ind_H^{\cB_N}(T))|_{\cB_N}|_{\cB_H}\succeq \Ind_H^{\cB_N}(T)|_{\cB_H}\succeq T
	\end{equation*}
	and $\Ind_H^\cB(T)|_{\cB_H}\succeq T$ by transitivity.
	
	To prove the second claim we take non degenerate representations $S$ and $T$ of $\cB_H$ and  $\cB,$ respectively, with faithful integrated forms.
	We have $\Ind_H^\cB(S)\preceq T,$ so $S\preceq \Ind_H^\cB(S)|_{\cB_H}\preceq T|_{\cB_H}$  and  $\intform{T|_{\cB_H}}$ must be faithful because $\intform{S}$ is.
\end{proof}

Take for $G$ the free group in three generators, $\bF_3=<a,b,c>,$ and think of $\bF_2=<a,b>$ as a subgroup of $\bF_3.$
Let $\cT=\{\bC\delta_t\}_{t\in \bF_3}$ be the trivial Fell bundle and define 
\begin{equation*}
\cB:=\{ 0\delta_t \}_{t\in \bF_3\setminus \bF_2}\cup \{\bC\delta_t\}_{t\in \bF_2}\subset \cT.
\end{equation*}
Then $\cB$ is a Fell bundle with the structure inherited from $\cT.$

Non degenerate representations of $\cB$ and unitary representations of $\bF_2$ are in one to one correspondence via an association 
\begin{equation*}
(T\colon \cB\to \bB(X))\leftrightsquigarrow (U^T\colon \bF_2\to \bB(X))  
\end{equation*}
where $U^T_t:= T_{ 1 \delta_t }.$
This identification preserves direct sums, weak containment and unitary equivalence.

Given a non degenerate representation $T\colon \cB\to \bB(X)$ and $f\in C_c(\cB),$ let $f'\in C_c(\bF_2)$ be such that $f(t)=f'(t)\delta_s$ for all $t\in \bF_2.$
It is easy so show $\intform{T}_f = \intform{U}^T_{f'}$ and this implies the existence of a unique isomorphism $\pi\colon C^*(\cB)\to C^*(\bF_2)$ extending $C_c(\cB)\to C_c(\bF_2),$ $f\mapsto f'.$ 
Hence, $C^*(\cB)=C^*(\bF_2).$

Fix $r\in \bF_3$ and set $H:=t\bF_2\tmu.$
We want to identify $C^*_{H\uparrow\bF_3}(\cB)$ and $C^*_{H\uparrow \cB}(\cB).$
According to remark \ref{rem: CHGB constant by conjutation}, $C^*_{H\uparrow\bF_3}(\cB)=C^*_{\bF_2\uparrow\bF_3}(\cB)$ and to compute this last algebra we take non degenerate faithful representations $\intform{T}\colon C^*(\cB)\to \bB(X)$ and $\intform{V}\colon C^*(\bF_2)\to \bB(Y).$
We have
\begin{equation*}
U^{T\otimes \Ind_{\bF_2}^{\bF_3}(V)}\cong U^T\otimes \Ind_{\bF_2}^{\bF_3}(V)|_{\bF_2}.
\end{equation*}
Since $V$ weakly contains the trivial representation of $\bF_2$ and $V\preceq \Ind_{\bF_2}^{\bF_3}(V)|_{\bF_2}$ (lemma \ref{lem:wc induction and restriction}), we have  $U^T\preceq U^T\otimes V\preceq U^T\otimes \Ind_{\bF_2}^{\bF_3}(V)|_{\bF_2}.$
Thus the integrated form of $U^{T\otimes \Ind_{\bF_2}^{\bF_3}(V)}$ is faithful and, consequently, $ T\intform{\otimes}\Ind_{\bF_2}^{\bF_3}(V) $ is a faithful representation of $C^*(\cB)$ that factors through a representation of $C^*_{\bF_2\uparrow \bF_3}(\cB)$ via $q^{\bF_2\uparrow \bF_3}_\cB\colon C^*(\cB)\to C^*_{\bF_2\uparrow \bF_3}(\cB).$
It is then clear that $q^{\bF_2\uparrow \bF_3}_\cB$ is faithful and we may think $C^*_{\bF_2\uparrow \bF_3}(\cB) =C^*(\bF_2).$

To compute $C^*_{H\uparrow\cB}(\cB)$ we must consider the representations of $\cB_H.$
The only non zero fibers of $\cB_H$ are those over $H\cap \bF_2.$
We divide the discussion in two cases: $t\in \bF_2$ and $t\notin\bF_2.$
In the first case, $\cB_H$ is the trivial bundle over the group $H=\bF_2$  and representations of $\cB,$ $\cB_H$ and $\bF_2$ are in one to one correspondence $T\leftrightsquigarrow T|_{\cB_H}\leftrightsquigarrow U^T.$
Moreover, this association preserves weak containment and direct sums.
By Lemma \ref{lem:wc induction and restriction}, for every representation $S$ of $\cB_H$ we have $S\preceq \Ind_H^\cB(S)|_{\cB_H}$ so $\intform{\Ind}_H^\cB(S)$ is faithful if $\intform{S}$ is.
By proposition \ref{prop: CHBB universal property}, $q^{H\uparrow\cB}\colon C^*(\cB)\to C^*_{H\uparrow\cB}(\cB)$ is an isomorphism and thus we get $C^*_{H\uparrow\cB}(\cB) = C^*(\bF_2).$

Now assume $t\notin \bF_2,$ in which case the word $t$ contains a letter $c$ or $c^{-1}$ and this implies $H\cap \bF_2=\{e\}.$
The only non zero fiber of $\cB_H$ is then $B_e=\bC\delta_e$ and for every representation $T$ of $\cB_H$ we have $T\cong \Ind_{\{e\}}^{\cB_H}(T|_{B_e}).$
By induction in stages \cite[XI 12.15]{FlDr88} and proposition \ref{prop: CHBB universal property}, $C^*_{H\uparrow \cB}(\cB)=C^*_\red(\cB).$
For the identity $\phi\colon B_e\to \bC,$ $U^{\Ind_{\{e\}}^{\cB}(\phi)}$ is the regular representation of $\bF_2$ and thus we may think of $q^{H\uparrow \cB}\colon C^*(\cB)\to C^*_{H\uparrow\cB}(\cB)$ as the regular representation $\regrep{\bF_2}\colon C^*(\bF_2)\to C^*_\red(\bF_2).$

\subsection{Exotic coactions}

Let $G$ be a LCH group and fix $H\leqslant G.$
For the trivial Fell bundle $\cT_G=\{\bC\delta_t\}_{t\in G}$ we have $C^*(G)=C^*(\cT_G).$
Since $\cT_G$ is saturated, we have a *-homomorphism $\pi^{\cT_G}_H$ as in \eqref{equ: pi uparrow H}.
We know $\pi^{\cT_G}_H$ is an isomorphism if $H$ is normal.
The reason for this is, basically, that the trivial representation $\tau_H\colon H\to \bC$ is weakly contained in $\Ind_H^G(\tau_H)|_H,$ which also happens if the normalizer of $H$ is open in $G$ (lemma \ref{lem:wc induction and restriction}).
In \cite{Derighetti} Derighetti gives a number of conditions implying $\tau_H\preceq \Ind_H^G(\tau_H)|_H.$

\begin{remark}
	If there exists a representation $V$ of $G$ such that $\tau_H\preceq V|_H,$ then $\pi^{\cT_G}_H$ is an isomorphism.
	Indeed, for every representation $U$ of $H$ with faithful integrated form, $\Ind_H^G(U)\preceq \Ind_H^G( V|_H\otimes U )\cong V\otimes \Ind_H^G(U).$
	Hence, $\|q^{H\uparrow\cT_G}(f)\|\leq \| [V\otimes \Ind_H^G(U)]_f \|\leq \|q^{H\uparrow G}_{\cT_G}(f)\|=\|\pi^{\cT_G}_H(q^{H\uparrow \cT_G}(f))\|$ for all $f\in C^*(G),$ implies the desired result.
\end{remark}

\begin{definition}
	We say $G$ satisfies Derighetti's weak condition with respect to $H\leqslant G$ if for some representation $U$ of $H$ it follows that $\tau_H\preceq \Ind_H^G(U)|_H,$ 
\end{definition}

We now focus our attention on the quotients $Q^G_H:=C^*_{H\uparrow G}(\cT_G)$ of $C^*(G)$ and the  natural quotient maps $q^G_H\equiv q^{H\uparrow G}_{\cT_G}\colon C^*(G)\to Q^H_G.$

\begin{theorem}\label{thm: exotic coaction}
	If $\cB=\{B_t\}_{t\in G}$ is a Fell bundle and  $\delta\colon C^*(\cB)\to M(C^*(\cB)\otimes Q^G_H)$ the *-homomorphism corresponding to $q^G_H$ (see the introduction), then there exists a unique *-homomorphism $\rho \colon C^*_{H\uparrow G}(\cB)\to M(C^*_{H\uparrow G}(\cB)\otimes Q^G_H)$ making 
	\begin{equation*}
	\xymatrix{ C^*(\cB) \ar[rr]^{\delta} \ar[d]_{q^{H\uparrow G}_\cB} & & M(C^*(\cB)\otimes Q^G_H)\ar[d]^{\overline{q^{H\uparrow G}_\cB\otimes 1}} \\   	
		C^*_{H\uparrow G}(\cB)\ar[rr]_\rho	& & M(C^*_{H\uparrow G}(\cB)\otimes Q^G_H) }
	\end{equation*}
	a commutative diagram. 
	Moreover, $\rho$ is faithful if $G$ satisfies Derighetti's weak condition with respect to $H.$
\end{theorem}
\begin{proof}
	Let $\intform{T}\colon C^*(\cB)\to \bB(X),$ $\intform{U}\colon C^*(H)\to \bB(Y)$ and $\intform{V}\colon C^*(G)\to \bB(Z)$ be non degenerate faithful representations.
	By proposition \ref{prop: CHGB universal property}, there are unique faithful representations $\mu\colon C^*_{H\uparrow G}(\cB)\to \bB( X\otimes Y )$ and $\nu\colon Q^G_H\to \bB(Z\otimes Y)$ such that $\mu\circ q^{H\uparrow G}_\cB = T\intform{\otimes}\Ind_H^G(U)$ and $\nu\circ q^G_H = V\intform{\otimes}\Ind_H^G(U).$
	We make extensive use of the faithful non degenerate representation $\mu\otimes \nu\colon C^*_{H\uparrow G}(\cB)\otimes Q^G_H\to \bB(X\otimes Y\otimes Z\otimes Y)$ and of its extension $\overline{\mu\otimes \nu}$ to the multiplier algebra.
	
	We have $(\mu\otimes \nu)\circ (q^{H\uparrow G}_\cB\otimes 1)= (T\intform{\otimes} \Ind_H^G(U))\otimes \mu$ and this implies $\overline{(\mu\otimes \nu)}\circ \overline{(q^{H\uparrow G}_\cB\otimes 1)}\circ \delta$ is the integrated form of
	\begin{align*}
	(T\otimes \Ind_H^G(U))\otimes (V\otimes \Ind_H^G(U)) & \cong T\otimes V \otimes \Ind_H^G(\Ind_H^G(U)|_H\otimes U)\\
	& \cong T\otimes \Ind_H^G(V|_H\otimes \Ind_H^G(U)|_H\otimes U).
	\end{align*}
	It then follows from proposition \ref{prop: CHGB universal property} that for all $f\in C^*(\cB)$ we have
	\begin{equation}\label{equ: exotic coaction}
	\| \overline{(q^{H\uparrow G}_\cB\otimes 1)}\circ \delta(f) \| = \|\overline{(\mu\otimes \nu)}\circ \overline{(q^{H\uparrow G}_\cB\otimes 1)}\circ \delta(f)\|\leq \|q^{H\uparrow G}_\cB(f)\|,
	\end{equation}
	and the existence of $\rho$ follows (uniqueness being immediate).
	
	Assume $G$ satisfies Derighetti's weak condition with respect to $H$ and let $W$ be a representation of $H$ such that $\tau_H\preceq \Ind_H^G(W)|_H.$
	Then $W\preceq U\Rightarrow \Ind_H^G(W)\preceq \Ind_H^G(U)\Rightarrow \tau_H\preceq \Ind_H^G(W)|_H\preceq \Ind_H^G(U)|_H\Rightarrow \tau_H\preceq \Ind_H^G(U)|_H.$
	In addition, $\tau_H\preceq V|_H$ because $\Ind_H^G(U)\preceq V.$
	Hence, $U\cong \tau_H\otimes \tau_H\otimes U\preceq V|_H\otimes \Ind_H^G(U)|_H\otimes U$
	and the inequality of \eqref{equ: exotic coaction} becomes an equality (implying $\rho$ is isometric).   
\end{proof}

Let $\rho^\cB\colon C^*_{H\uparrow G}(\cB)\to M(C^*_{H\uparrow G}\otimes Q^G_H) $ and $\rho^{HG}\colon Q^G_H\to M(Q^H_H\otimes Q^G_H)$ be the *-homomorphisms given by the theorem above.
To prove the coaction identity $\overline{(\rho^\cB \otimes \iota )}\circ \rho^\cB = \overline{(\iota  \otimes \rho^{HG} ) }\circ \rho^\cB $ one may take faithful non degenerate representations $\intform{T},$ $\intform{U}$ and $\intform{V}$ of $C^*(\cB),$ $C^*(H)$ and $C^*(G),$ respectively, and consider the (faithful) representations $\mu$ and $\nu $ of $C^*_{H\uparrow G}(\cB)$ and $Q^G_H,$ respectively, such that $\mu\circ q^{H\uparrow G}_\cB=T\intform{\otimes} \Ind_H^G(U)$ and $\nu\circ q^G_H = V\intform{\otimes} \Ind_H^G(U).$
The extension $\overline{(\mu\otimes \nu\otimes \nu)}$ is a faithful representation of $M(C^*_{H\uparrow G}(\cB)\otimes Q^G_H\otimes Q^G_H)$ and $\overline{(\mu\otimes \nu\otimes \nu)} \circ \overline{(\rho^\cB \otimes \iota)}\circ \rho^\cB\circ  q^{H\uparrow G}_\cB$ is the integrated form of 
\begin{equation*}
[T\otimes \Ind_H^G(U)]\otimes [ V\otimes \Ind_H^G(U) ]\otimes [ V\otimes \Ind_H^G(U) ],
\end{equation*}
as that of $\overline{(\mu\otimes \nu\otimes \nu)}\circ \overline{(\iota \otimes \rho^{HG} ) }\circ \rho^\cB\circ q^{H\uparrow G}_\cB$ is,
implying 
\begin{align*}
\overline{(\mu\otimes \nu\otimes \nu)} \circ \overline{(\rho^\cB \otimes \iota)}\circ \rho^\cB\circ  q^{H\uparrow G}_\cB = \overline{(\mu\otimes \nu\otimes \nu)}\circ \overline{(\iota \otimes \rho^{HG} ) }\circ \rho^\cB\circ q^{H\uparrow G}_\cB.
\end{align*}

The identity above implies $\overline{(\rho^\cB \otimes \iota )}\circ \rho^\cB = \overline{(\iota  \otimes \rho^{HG} ) }\circ \rho^\cB $ because $\overline{(\mu\otimes \nu\otimes \nu)}$ is faithful and $q^{H\uparrow G}_\cB$ surjective.

\subsection{Inner amenable groups}

Buss, Echterhof and Willet have asked \cite[Question 8.8]{buss2020amenability} for a class $\cC$ of groups for which the nuclearity of $A\rtimes_{\red \alpha }G$ implies the amenability of the action $\alpha.$
McKee and Pourshashami \cite{mckee2020amenable} showed $\cC$ contains all inner amenable groups.

We are not dealing with amenability here, but with weak containment.
So it is natural to replace nuclearity of $A\rtimes_{\red \alpha }G$ with something weaker. 
We propose $(A\rtimes_{\red \alpha }G)\otimes_{\max} C^*_\red(G) =(A\rtimes_{\red \alpha }G)\otimes C^*_\red(G) $ and ask for which class $\cC$ of groups this condition implies $A\rtimes_\alpha G=A\rtimes_{\red \alpha }G.$
In Corollary \ref{cor:inner amenability and weak containment} we show $\cC$ contains all inner amenable groups.

Following the ideas presented in the introduction, given a Fell bundle $\cB=\{B_t\}_{t\in G}$ one can construct a *-homomorphism
\begin{equation*}
\Phi\colon C^*(\cB)\to M(C^*(\cB)\otimes_{\max} C^*(G))
\end{equation*}
such that given a non degenerate representations $\pi\colon C^*(\cB)\otimes_{\max} C^*(G)\to \bB(X),$ $\overline{\pi}\circ\Phi$ is the integrated form of $TU\colon \cB\to \bB(X),$ $(b\in B_t)\mapsto T_bU_t,$ with the integrated forms of $T\colon \cB\to \bB(X)$ and $U\colon G\to \bB(X)$ being $f\mapsto \overline{\pi}(f\otimes_{\max} 1)$ and $g\mapsto\overline{\pi}(1\otimes_{\max} g),$ respectively.
One may arrange $\pi$ so that $U\colon G\to \bB(X)$ is trivial ($t\mapsto 1$) and $\intform{T}$ is faithful, so $\overline{\pi}\circ \Phi=\intform{T} $  and it follows that $\Phi$ is faithful.

It is shown in \cite{ExlibroAMS} that if $G$ is discrete, then the diagonal arrow of the commutative diagram
\begin{equation*}
\xymatrix{ C^*(\cB)\ar[rr]^\Phi\ar[drr]_{\Phi^r} & & M(C^*(\cB)\otimes_{\max} C^*(G))\ar[d]^{\overline{\regrepEN{\cB} \otimes_{\max} \intform{\lambda} }}\\ & &M(C^*_\red(\cB)\otimes_{\max} C^*_\red(G)) }
\end{equation*}
is faithful.
This is not an exclusive property of discrete groups.

\begin{proposition}
	If $G$ is inner amenable, then $\Phi^\red$ is faithful.
\end{proposition}
\begin{proof}
	Let $\intform{T}\colon C^*(\cB)\to \bB(X)$ a non degenerate faithful representation.
	The left and right regular representations $\lambda,\rho\colon G\to \bB(L^2(G))$ commute and are unitary equivalent, so the ranges of $T\otimes \rho$ and $1\otimes\lambda $ commute and their integrated forms combine to produce a representation
	\begin{equation*}
	\pi\colon C^*_\red(\cB)\otimes_{\max} C^*_\red(G)\to \bB(X\otimes L^2(G))
	\end{equation*}
	
	The composition $\overline{\pi}\circ \overline{\regrepEN{\cB} \otimes_{\max} \intform{\lambda} }\circ \Phi = \overline{\pi}\circ \Phi^\red$ is the integrated form of $T\otimes \omega$ with $\omega\colon G\to \bB(L^2(G))$ given by $\omega_t =\lambda_t\rho_t.$
	To say $G$ is inner amenable is equivalent to say $\omega$ weakly contains the trivial representation of $G,$ so $T\preceq T\otimes \omega$ and it follows that $\overline{\pi}\circ \Phi^\red =T\intform{\otimes}\omega$ is faithful.
	Hence, $\Phi^r$ is faithful.	
\end{proof}

\begin{corollary}\label{cor:inner amenability and weak containment}
	If $G$ is inner amenable and $C^*_\red(\cB)\otimes_{\max} C^*_\red(G)=C^*_\red(\cB)\otimes C^*_\red(G),$ then $C^*(\cB)=C^*_\red(\cB).$
\end{corollary}
\begin{proof}
	For $H=\{e\}$ the *-homomorphism $\delta\colon C^*(\cB)\to \bB( C^*(\cB)\otimes Q^G_H)$ we use to define $C^*_{H\uparrow G}(\cB)=\delta(C^*(\cB))$ is $\Phi^\red$ and $C^*_{H\uparrow G}(\cB)=C^*_\red(\cB).$
	But $\Phi^\red$ is faithful, so  $C^*(\cB)=C^*_\red(\cB).$
\end{proof}

\section{An absorption principle}\label{sec: abs prin}

Fell's absorption principle does not hold for non saturated Fell bundles.
Indeed, the example of section \ref{ssec:example} reveals $ T\intform{\otimes} \Ind_{r\bF_2 \rmu}^\cB(V)$ may be a faithful representation of $C^*(\cB)$ while $\Ind_{r\bF_2\rmu}^\cB(T|_{\cB_{r\bF_2\rmu}}\otimes V)$ factors through a faithful representation of $C^*_\red(\cB)\neq C^*(\cB).$
Therefore, $\Ind_{r\bF_2\rmu}^\cB(T|_{\cB_{r\bF_2\rmu}}\otimes V)$ can not be unitary equivalent to $T\intform{\otimes} \Ind_{r\bF_2\rmu}^\cB(V)$ (not even weakly equivalent).

The theorem we present below is our best answer to question \eqref{item: q general abs prin} of the Introduction on Fell's principle.
It emerged during our attempt to generalize Fell's and Exel-Ng's absorption principles.
The computations revealed that the ``moving around'' technique of \cite[pp 515]{ExNg} is appropriate when the inducing subgroup $H$ is stable by conjugation (normal).
In the general case, conjugation ``moves $H$ around'' and this is the reason why the conjugated subgroups $tH\tmu$ appear bellow.
The proof is quite technical, we suggest to skip it on a first read.

\begin{theorem}\label{thm:Fells absorption principle I}
	Let $\cB=\{B_t\}_{t\in G}$ be a Fell bundle,  $H\leqslant G,$ $T\colon \cB\to \bB(X)$ a non degenerate representation and $U\colon H\to \bB(Y)$ a unitary representation.
	If for each $t\in G$  we denote ${}_tU$ the conjugated representation $tH\tmu \to \bB(Y), \ r\mapsto U_{\tmu r t},$ then 
	\begin{equation*}
	T\otimes \Ind_H^G(U)\approx \{\Ind_{tH\tmu }^{\cB}(T|_{\cB_{tH\tmu}}\otimes {}_tU) \}_{t\in G}.
	\end{equation*}
\end{theorem}

\begin{proof}
	For convenience we introduce the following notation 
	\begin{align*}
	{}_tH&\equiv tH\tmu & {}_{t|}T & \equiv T|_{\cB_{tH\tmu}} & V&\equiv \Ind_H^G(U) & p_t&:=p_{tH\tmu}.
	\end{align*}
	
	The Hilbert space induced by $U$ is $Y^p_U$ and we regard $X\otimes^G Y^p_U:=\ell^2(G)\otimes X\otimes Y^p_U$ as an $\ell^2-$direct sum of $\# G$ copies of $X\otimes Y^p_U$.
	Similarly, the direct sum of $\#G$ copies of $T\otimes V$ is 
	\begin{align*}
	T\otimes^G V &\colon \cB\to \bB(X\otimes^G Y^p_U) & b&\mapsto 1_{\ell^2(G)}\otimes [T\otimes V]_b,
	\end{align*}
	which is the composition of $T\otimes V$ with the faithful *-homomorphism
	\begin{equation*}
	\Theta\colon \bB(X\otimes Y^p_U)\to \bB(X\otimes^G Y^p_U),\ \Theta(R)=1_{\ell^2(G)}\otimes R. 
	\end{equation*}

	If we denote $T\intform{\otimes}V$ and $T\intform{\otimes}^G V$ the integrated forms of $T\otimes V$ and $T\otimes^G V,$ respectively, then $\Theta\circ (T\intform{\otimes} V)=T\intform{\otimes}^GV.$
	Indeed, for all $f\in C_c(\cB)$ and every elementary tensor $g\otimes \xi\otimes \eta \in \ell^2(G)\otimes X\otimes Y^p_U$ we have
	\begin{align*}
	[T\intform{\otimes}^GV]_f(g\otimes \xi\otimes \eta) & = \int_G g\otimes [T\otimes V]_{f(t)}(\xi\otimes \eta)\, dt  
	=g\otimes \left([T\intform{\otimes}V]_f (\xi\otimes\eta)\right)\\
	& = \Theta\circ [T\intform{\otimes} V]_f(g\otimes \xi\otimes \eta).
	\end{align*}

	Since $\Theta$ is an isometry and both $T\otimes V$ and all the members of $\{\Ind_{{}_t H}^{\cB}({}_{t|}T\otimes {}_tU)\}_{t\in G}$ are non degenerate, we can use \cite[XI 8.20]{FlDr88} to the deduce it suffices to show that 
	\begin{equation}\label{equ:equivalent condition of thesis for absorption principle I}
	\|[T\intform{\otimes}^G V]_f  \|= \sup\{ \|\intform{\Ind}_{{}_t H}^{\cB}({}_{t|}T\otimes {}_tU)_f\|\colon t\in G \}\qquad \forall\ f\in C_c(\cB).
	\end{equation}
	
	The $\ell^2-$direct sum 
	\begin{equation}\label{equ:direct sum of induced}
	S:=\bigoplus_{t\in G} \Ind_{{}_tH}^\cB({}_{t|}T\otimes {}_tU)
	\end{equation}
	is a non degenerate representation whose integrated form is $\intform{S}= \oplus_{t\in G} \intform{\Ind}_{{}_tH}^\cB({}_{t|}T\otimes {}_tU)$ and $\|\intform{S}_f\|=\sup\{ \|\intform{\Ind}_{{}_t H}^{\cB}({}_{t|}T\otimes {}_tU)_f\|\colon t\in G \}.$
	Thus $S$ is weakly equivalent to $T\otimes^G V$ if and only if \eqref{equ:equivalent condition of thesis for absorption principle I} holds.
	
	To compare $T\otimes^G V$ and $S$  we will construct a linear isometry between their domains, 
	\begin{equation}\label{equ:domain and range of I}
	I\colon \bigoplus_{t\in G} (X\otimes Y)^{p_t}_{{}_{t|}T\otimes {}_tU} \to X\otimes^G Y^p_U,
	\end{equation}
	that intertwines $S$ and $T\otimes^GV.$
	To state the properties defining $I$ we need to prove the existence of a (unique) linear function 
	\begin{equation}\label{equ:map L}
	L\colon C_c(G,X\otimes^G Y )\to X\otimes^G   Y^p_U
	\end{equation}
	which is continuous in the inductive limit topology and satisfies
	\begin{equation}\label{equ:defining condition for L}
	L( f\odot [\delta_t\otimes \xi\otimes \eta ]) =\delta_t\otimes ( \xi \otimes [f\otimes_{U} \eta])
	\end{equation}
	for all $f\in C_c(G),$ $\xi\in X,$ $\eta\in Y$ and $t\in G,$ where we regard the algebraic tensor product $C_c(G)\odot (X\otimes^G Y)$ as a subspace of $C_c(G,X\otimes^G Y)$ in the usual way ($f\odot u$ is the function $t\mapsto f(t)u$).
	
	Uniqueness of $L$ follows from the fact that the functions of the form $f\odot (\delta_t\otimes \xi\otimes \eta )$ span a subset of $C_c(G,X\otimes^G Y)$ which is dense in the inductive limit topology~\cite[II 14.6]{FlDr88}.
	To prove the existence of $L$ we take $u,v\in C_c(G,X\otimes^G Y)$ such that $u = f\odot (\delta_r\otimes \xi\otimes \eta )$ and $v=g\odot (\delta_s\otimes \zeta\otimes \kappa)$ for some  $f,g\in C_c(G),$ $r,s\in G,$ $\xi,\zeta\in X$ and $\eta,\kappa\in Y.$
	The inner product of the candidates for $L(u)$ and $L(v)$ is
	\begin{align}
	\langle & \delta_r\otimes \xi \otimes [f\otimes_{U} \eta],\delta_s\otimes \zeta \otimes [g \otimes_{U} \kappa]\rangle \nonumber\\
	& =  \langle \delta_r,\delta_s\rangle\langle \xi,\zeta\rangle \langle f\otimes_{U} \eta,g \otimes_{U} \kappa\rangle \nonumber \\
	& = \langle \delta_r,\delta_s\rangle \langle \xi,\zeta\rangle \int_{H} \int_G \Delta_H^G(t)^{1/2} \overline{f(z)} g(zt) \langle \eta, U_t  \kappa\rangle \,dzdt \nonumber \\
	& = \int_{H}  \int_G \Delta_H^G(t)^{1/2} \overline{f(z)} g(zt) \langle \delta_r \otimes \xi\otimes\eta, (1\otimes 1\otimes U_t)(\delta_s\otimes \zeta\otimes \kappa )\rangle \,dzdt\nonumber \\
	&= \int_{H}  \int_G \Delta_H^G(t)^{1/2} \langle u(z), (1\otimes 1\otimes U_t)v(zt)\rangle \,dzdt.\label{equ:basic inner product identity to construct L}
	\end{align}
	
	Fix a compact $D\subset G$ and denote $C_D(G,X\otimes^G Y)$ the set formed by those $f\in C_c(G,X\otimes^G Y)$ with $\supp(f)\subset D.$
	This vector space is a Banach space with the norm $\|\ \|_\infty.$
	Let $C_D^\odot$ be the subspace of $C_D(G,X\otimes^G Y)$ spanned by the functions of the form $f\odot (\delta_t\otimes \xi\otimes \eta)$ with $f\in C_D(G),$ $t\in G,$ $\xi\in X$ and $\eta\in Y.$
	We clearly have $C_c(G)C_D^\odot \subset C_D^\odot.$
	If for each $t\in G$ we define $C_D^\odot(t)$ as the closure of $\{u(t)\colon u\in C_D^\odot\},$ then $C_D^\odot(t)=X\otimes^G Y$ for every $t$ in the interior of $D$ and $\{0\}$ otherwise.
	By \cite[Lemma 5.1]{Ab03}, the closure of $C_D^\odot$ in $C_c(G,X\otimes^G Y)$ with respect to the inductive limit topology is $\{f\in C_c(G,X\otimes^G Y)\colon f(t)\in C_D^\odot(t)\, \forall\, t\in G  \}=C_D(G,X\otimes^G Y).$
	
	Take any $u,v\in C_D(G,X\otimes^G Y).$
	By the preceding paragraph there are sequences $\{u_n\}_{n\in \bN}$ and $\{v_n\}_{n\in \bN}$ in $C_D^\odot$ converging uniformly to $u$ and $v,$ respectively.
	Then for all $n\in \bN$ there exists a positive integer $m_n$ and (for each $j=1,\ldots,m_n$) elements $f_{n,j},g_{n,j}\in C_D(G),$ $r_{n,j},s_{n,j}\in G,$ $\xi_{n,j},\zeta_{n,j}\in X$ and $\zeta_{n,j},\kappa_{n,j}\in Y$ such that
	\begin{align*}
	u_n &=\sum_{j=1}^{m_n} f_{n,j}\odot (\delta_{r_{n,j}}\otimes \xi_{n,j}\otimes \eta_{n,j} ) & v_n&=\sum_{j=1}^{m_n}g_{n,j}\odot (\delta_{s_{n,j}}\otimes \zeta_{n,j}\otimes \kappa_{n,j}).
	\end{align*}
	
	By \eqref{equ:basic inner product identity to construct L}, for all $a,b\in \bN$ we have
	\begin{align*}
	\|\sum_{j=1}^{m_a} \delta_{r_{a,j}} & \otimes (\xi_{a,j}\otimes [f_{a,j}\otimes_{U} \eta_{a,j}] )-\sum_{k=1}^{m_b}\delta_{s_{b,k}}\otimes ( \zeta_{b,k}\otimes [g_{b,k}\otimes_{U} \kappa_{b,k}] )\|^2 =\\
	& = | \int_{H}  \int_G \Delta_H^G(t)^{1/2} \langle (u_a-v_b)(z), (1\otimes 1\otimes U_t)(u_a-v_b)(zt)\rangle \,dzdt |.
	\end{align*}
	For the inner product inside the integral above to be non zero we must have $z,tz\in D,$ which implies $t= tz z^{-1}\in H\cap (DD^{-1}).$
	If $\alpha_D$ and $\beta_D$ are the measures of $D$ and $H\cap (D^{-1}D)$ with respect to the Haar measures of $G$ and $H,$ respectively, and $\gamma_D:=\sup\{ \Delta_H^G(t)^{1/2} \colon t\in H\cap (D^{-1}D)\},$ then 
	\begin{equation*}
	\|\sum_{j=1}^{m_a} \delta_{r_{a,j}} \otimes (\xi_{a,j}\otimes [f_{a,j}\otimes_{U} \eta_{a,j}] )-\sum_{k=1}^{m_b}\delta_{s_{b,k}}\otimes ( \zeta_{b,k}\otimes [g_{b,k}\otimes_{U} \kappa_{b,k}] )\|^2
	\end{equation*}
	is not greater than $\|u_a-v_b\|_\infty^2 \alpha_D\beta_D\gamma_D.$
	
	Several conclusion arise from that bound:
	\begin{enumerate}
		\item If we take $u=v$ and $v_b=u_b,$ it follows that $\{\sum_{j=1}^{m_a} \delta_{r_{a,j}} \otimes (\xi_{a,j}\otimes [f_{a,j}\otimes_{U} \eta_{a,j}]) \}_{a\in \bN}$ is a Cauchy sequence in $X\otimes^G Y^p_U;$ the limit of which we denote $L_D(\{u_n\}_{n\in \bN}).$
		\item If $u=v$ and we take limit in $a$ and $b,$ we get $L_D(\{u_n\}_{n\in \bN})= L_D(\{v_n\}_{n\in \bN}).$
		Thus we may define a function $L_D\colon C_D(G,X\otimes^G Y)\to X\otimes^G Y^p_U,\ u\mapsto L_D(u):=L_D(\{u_n\}_{n\in \bN}).$
		\item Taking limit in $a$ and $b$ we obtain $\| L_D(u)-L_D(v) \|\leq\|u-v\|_{\infty} \sqrt{\alpha_D\beta_D\gamma_D},$ so $L_D$ is continuous.
		\item $L_D(f\odot (\delta_t\otimes \xi\otimes\eta))=\delta_t\otimes \xi\otimes (f\otimes_U\eta)$ and $L_D$ is linear when restricted to $C_D^\odot.$
		Thus $L_D$ is linear.
		\item By~\eqref{equ:basic inner product identity to construct L} and the continuity of $L_D,$
		\begin{equation}\label{equ:L and inner product}
		\langle L_D(u),L_D(v)\rangle = \int_G \int_H \Delta_H^G(t)^{1/2}\langle u(s), (1\otimes 1 \otimes U_t)v(st)\rangle \,dt ds .
		\end{equation}
	\end{enumerate}
	
	It follows immediately that $L_E$ is an extension of $L_D$ whenever $E\subset G$ is a compact set containing $D.$
	Then there exists a unique function $L\colon C_c(G,X\otimes^G Y)\to X\otimes^G Y^p_U$ extending all the $L_D$'s.
	This extension is linear and continuous in the inductive limit topology by \cite[II 14.3]{FlDr88}.
	Note also that $L$ satisfies~\eqref{equ:defining condition for L} and, consequently, it has dense range.
	
	Now that we know $L$ exists we are one step closer of being able to specify the properties defining the map $I$ of \eqref{equ:domain and range of I}.
	Given $r\in G,$ $f\in C_c(\cB),$ $\xi\in X$ and $\eta\in Y$ we define $[r,f,\xi,\eta]\in C_c(G,X\otimes^G Y)$ by
	\begin{equation}\label{equ:right brakets for triple tensors}
	[r,f,\xi,\eta](s)= \Delta_G(r)^{-1/2} \delta_r\otimes T_{f(s\rmu)}\xi\otimes \eta.
	\end{equation}
	We claim there exists a unique linear and continuous map $I$ with the domains and ranges specified in~\eqref{equ:domain and range of I} and such that for all $t\in G,$ $f\in C_c(\cB),$ $\xi\in X$ and $\eta\in Y,$
	\begin{equation}\label{equ:defining property for I}
	I( f\otimes_{{}_{t|}T\otimes {}_tU} (\xi\otimes \eta) ) = L([t,f,\xi,\eta]).
	\end{equation}
	To prove this claim we start by taking elementary tensors $f\otimes_{{}_{r|}T\otimes {}_r U}(\xi\otimes \eta)$ and $g\otimes_{{}_{t|}T\otimes {}_tU}(\zeta\otimes\kappa)$ on (possibly equal) direct summands of $\bigoplus_{s\in G}  (X\otimes Y)^{p_s}_{{}_{s|}T\otimes {}_sU}.$
	Notice the $r$ and the $t$ in the subindexes indicate the direct summands the tensorss belongs to. 	
	By~\eqref{equ:L and inner product} we have
	\begin{align*}
	\langle & L([r,f,\xi,\eta]),L([t,g,\zeta,\kappa])\rangle= \\
	& = \int_G\int_H \Delta_H^G(w)^{1/2}\Delta_G(rt)^{-1/2} \langle \delta_r\otimes T_{f(z\rmu )}\xi \otimes \eta ,\delta_t\otimes T_{g(zw\tmu)}\zeta\otimes U_w\kappa\rangle  \,dwdz \\
	& = \int_H\int_G \Delta_H^G(w)^{1/2} \underbrace{\Delta_G(rt)^{-1/2}\langle \delta_r,\delta_t\rangle}_{=\Delta_G(\rmu)\langle \delta_r,\delta_t\rangle} \langle T_{f(z\rmu )}\xi \otimes \eta ,T_{g(zw\rmu)}\zeta\otimes U_w\kappa\rangle  \,dzdw \\
	& = \int_H\int_G \Delta_H^G(w)^{1/2} \langle \delta_r,\delta_t\rangle \langle T_{f(z)}\xi \otimes \eta ,T_{g(zr w\rmu)}\zeta\otimes U_w\kappa\rangle  \,dzdw\\
	&= \int_H\int_G \Delta_{{}_rH}^G(rw\rmu)^{1/2} \langle \delta_r,\delta_t\rangle \langle \xi \otimes \eta ,T_{f(z)^*g(zr w\rmu )}\zeta\otimes {}_r U_{rw\rmu}\kappa\rangle  \,dzdw \\
	&= \int_{{}_rH}\int_G \Delta_{{}_rH}^G(w)^{1/2} \langle \delta_r,\delta_t\rangle \langle \xi \otimes \eta ,T_{f(z)^*g(zw)}\zeta\otimes {}_r U_{w}\kappa\rangle  \,dzdw \\
	&= \langle \delta_r,\delta_t\rangle \langle \xi \otimes \eta ,[{}_{r|}T\intform{\otimes} {}_r U]_{p_r(f^**g)}(\zeta\otimes \kappa)\rangle  \\
	&= \langle f\otimes_{{}_{r|}T\otimes {}_r U}(\xi\otimes \eta),g\otimes_{{}_{t|}T\otimes {}_tU}(\zeta\otimes\kappa)\rangle.
	\end{align*}
	The existence of $I$ then follows immediately from remark~\ref{thm:Fells absorption principle I} if one considers the functions
	\begin{align*}
	b&\colon G\times C_c(\cB)\times X\times Y\to \bigoplus_{t\in G}(X\otimes Y)^{p_t}_{{}_{t|}T\otimes{}_tU} & b(r,f,\xi,\eta)&=f\otimes_{{}_{r|}T\otimes {}_r U}(\xi\otimes \eta)\\
	c& \colon G\times C_c(\cB)\times X\times Y\to X\otimes^G Y^p_U & c(r,f,\xi,\eta)&=L([r,f,\xi,\eta]).
	\end{align*}
	
	We are not sure if $I$ is surjective, but we can use the ideas of~\cite{ExNg} to ``move'' the image of $I$ to ``fill'' $X\otimes^G Y.$
	The movement is via the map
	\begin{align*}
	\rho & \colon G\to \bB(X\otimes^G Y^p_U) & \rho_t &:= \lt_t\otimes 1_{X}\otimes 1_{Y^p_U},
	\end{align*}
	where $\lt\colon G\to \bB(\ell^2(G))$ is the left regular representation of the discrete version of $G$ ($\lt_t(\delta_s)=\delta_{ts}$).
	Note $\rho$ and $\Theta$ have commuting ranges, so the range of $\rho$ commutes with that of $T\intform{\otimes}^GV=\Theta\circ (T\intform{\otimes} G).$ 	
	We remark that the continuity of $\rho$ (which may fail) plays no r\^ole in the proof.
	
	Let $K$ be the image of $I.$
	We claim that $ G\cdot K:= \spn\{\rho_t K\colon t\in G\}$
	is dense in $X\otimes^G Y^p_U.$
	To prove this we define, for each $t\in G,$ the function 
	\begin{align*}
	\mu_t \colon &  C_c(G,X\otimes^G Y)\to C_c(G,X\otimes^G Y) & (\mu_t f)(z) &=(\lt_t\otimes 1_{X}\otimes 1_{Y})f(z). 
	\end{align*}
	In particular, $\mu_t (f\odot (\delta_r\otimes \xi\otimes \eta))=f\odot (\delta_{tr}\otimes \xi\otimes \eta).$
	Hence,
	\begin{align*}
	L\circ \mu_t (f\odot (\delta_r\otimes \xi\otimes \eta)) 
	& =	L(f \odot (\delta_{tr}\otimes \xi\otimes \eta))
	=\delta_{tr}\otimes \xi \otimes [f\otimes_{U} \eta] \\
	& =\rho_t  (\delta_{r}\otimes \xi \otimes [f\otimes_{U} \eta])
	=\rho_t L(f\odot( \delta_r\otimes \xi\otimes \eta ))
	\end{align*}
	and we get $L\circ \mu_t =\rho_t \circ L$ because both $L\circ \mu_t $ and $\rho_t \circ L$ are linear and continuous in the inductive limit topology and agree on a dense set. 
	Thus $\overline{G\cdot K}$ contains the image through  $L$ of
	\begin{equation*}
	K_0:= \spn\{ \mu_t [r,f,\xi,\eta]\colon r,t\in G,\, f\in C_c(\cB),\, \xi\in X,\, \eta\in Y  \}\subset C_c(G,X\otimes^G Y).
	\end{equation*}
	
	Note $C(G)K_0\subset K_0.$
	Besides, 
	\begin{equation}\label{equ:mu []}
	\mu_r [t,f,\xi,\eta](z)
	=\Delta_G(t)^{-1/2}\delta_{rt}\otimes T_{f(z\tmu)} \xi\otimes \eta.
	\end{equation}
	Fixing $z\in G$ and varying $r,t\in G,$ $\xi\in X,$ $\eta\in Y$ and $f\in C_c(\cB),$ the elements we obtain on the right hand side of \eqref{equ:mu []} are all those of the form $\delta_{s}\otimes T_b \xi\otimes \eta,$ for arbitrary $s\in G,$ $b\in \cB,$ $\xi\in X$ and $\eta\in Y.$
	The closed linear span of this vectors is  $X\otimes^G Y$ because $T$ is non degenerate, and we conclude (using~\cite[II 14.3]{FlDr88}) that $K_0$ is dense in $C_c(G,X\otimes^G Y)$ in the inductive limit topology.
	Hence, $\overline{G\cdot K}$ contains the dense set $L(K_0)$ and it follows that $\overline{G\cdot K}=X\otimes^G Y^p_U.$
	
	Our next goal is to show that $I$ intertwines the $S$ of \eqref{equ:direct sum of induced} and $T\otimes^G V.$
	To prove this claim we fix $r,s,t\in G,$ $b\in B_r,$ $f\in C_c(\cB),$ $g\in C_c(G),$ $\xi,\zeta\in X$ and $\eta, \kappa\in Y.$
	For convenience we denote $u$ and $v$ the tensors $f\otimes_{{}_{s|}T\otimes {}_{s}U} (\xi\otimes \eta) $ and $g\odot (\delta_t\otimes \zeta\otimes \kappa),$ respectively.
	Recalling~\eqref{equ:L and inner product} we get
	\begin{align*}
	\langle & [T\otimes^G V]_b  \circ I (u ),L(v ) \rangle
	= \langle  L([s,f,\xi,\eta]),[T\otimes^G V]_{b^*}(\delta_t\otimes \zeta\otimes (g\otimes_{U}\kappa)) \rangle \nonumber\\
	&= \langle  L([s,f,\xi,\eta]),\delta_t\otimes T_{b^*}\zeta\otimes ( \rmu g\otimes_{U}\kappa ) \rangle \nonumber\\
	&= \langle  L([s,f,\xi,\eta]),L(\rmu g\odot( \delta_t \otimes T_{b^*}\zeta\otimes \kappa))\rangle\\
	& =\int_H\int_G  \Delta_H^G(w)^{1/2}  \langle [s,f,\xi,\eta](z),\delta_t \otimes T_{b^*}\zeta\otimes U_w\kappa \rangle g(r zw)\,dzdw\\
	& =\int_H\int_G  \Delta_H^G(w)^{1/2} \Delta_G(s)^{-1/2} \langle \delta_s\otimes T_{f(zs^{-1})}\xi\otimes \eta,\delta_t \otimes T_{b^*}\zeta\otimes U_w\kappa \rangle g(r zw)\,dzdw\\
	& =\int_H\int_G  \Delta_H^G(w)^{1/2} \Delta_G(s)^{-1/2}\langle \delta_s,\delta_t\rangle \langle T_{bf(\rmu zs^{-1})}\xi\otimes \eta,\zeta\otimes U_w\kappa \rangle g(zw)\,dzdw\\
	& =\int_H\int_G  \Delta_H^G(w)^{1/2} \Delta_G(s)^{-1/2}\langle \delta_s,\delta_t\rangle \langle T_{(bf)(zs^{-1})}\xi\otimes \eta,\zeta\otimes U_w\kappa \rangle g(zw)\,dzdw\\
	& =\langle L([s,bf,\xi,\eta]),L(g\odot (\delta_t\otimes \zeta\otimes \kappa))\rangle
	=\langle I(bf\otimes_{{}_{|s}T\otimes {}_s U}(\xi\otimes \eta)),L(v)\rangle\\
	&= \langle I\circ S_b(u),L(v)\rangle.
	\end{align*}
	Since $u$ and $v$ are arbitrary basic tensors and $L$ has dense range, by linearity and continuity we get that $[T\otimes^G V]_b \circ I = I\circ S_b;$ which implies that for all $f\in C_c(\cB)$
	\begin{equation}\label{equ:T,V,I,R}
	[T\intform{\otimes}^G V]_f \circ I = I\circ \intform{S}_f.
	\end{equation}
	The same identity holds for all $f\in C^*(\cB)$ because $C_c(\cB)$ is dense in $C^*(\cB).$
	
	Recall that  $G\cdot K$ spans a dense subset of $\ell^2(G)\otimes X\otimes Y^p_U, $ thus for all $g\in C^*(\cB) $ we have  
	\begin{align*}
	[T\intform{\otimes}^G V]_g=0
	& \Leftrightarrow \ [T\intform{\otimes}^G V]_g\rho_t \circ I=0\ \forall \ t\in G
	\Leftrightarrow \ \rho_t [T\intform{\otimes}^G V]_g\circ I=0\ \forall \ t\in G\\
	& \Leftrightarrow \ [T\intform{\otimes}^G V]_g\circ I=0
	\Leftrightarrow \ I\circ \intform{S}_g = 0
	\Leftrightarrow \intform{S}_g = 0,
	\end{align*}
	where the last equivalence follows from the fact that $I$ is an isometry.
	We may then define a *-homomorphism $\Omega \colon (T\intform{\otimes}^G V)(C^*(\cB))\to \intform{S}(C^*(\cB))$ by $\Omega\left([T\intform{\otimes}^G V]_g\right)=\intform{S}_g.$
	
	The thesis of the theorem was shown to be the equivalent to \eqref{equ:equivalent condition of thesis for absorption principle I} which, in turn, is equivalent to say $\Omega $ is an isometry.
	But $\Omega$ is indeed an isometry because it is an injective *-homomorphism between two C*-algebras.
\end{proof}

The first application of theorem \ref{thm:Fells absorption principle I} is the comparison of the C*-algebras $C^*_{H\uparrow G}(\cB)$ and $C^*_{H\uparrow\cB}(\cB)$ when the normalizer of $H$ is open in $G.$
This covers the cases examined in the example of section \ref{ssec:example}, where there is always a quotient map $\Theta^H_\cB\colon C^*_{H\uparrow G}(\cB)\to C^*_{H\uparrow \cB}(\cB)$ that it is not always injective.

\begin{corollary}\label{cor: maps Theta}
	Given a Fell bundle $\cB=\{B_t\}_{t\in G}$ and $H\leqslant G$ with open normalizer, there exists a *-homomorphism $\Theta^H_\cB\colon C^*_{H\uparrow G}(\cB)\to C^*_{H\uparrow \cB}(\cB)$ such that $\Theta^{\cB}_H\circ q^{H\uparrow G}_\cB=q^{H\uparrow \cB}.$
	If for all $r\in G$ we identify $C^*_{H\uparrow G}(\cB)=C^*_{rH\rmu\uparrow G}(\cB)$ as indicated in remark \ref{rem: CHGB constant by conjutation}, then $\cap_{r\in G} \ker(\Theta^{rH\rmu}_\cB)=\{0\}.$
	Moreover, $\Theta^H_\cB$ is an isomorphism if either $\cB$ is saturated or $H$ normal.
\end{corollary}
\begin{proof}
	Let $\intform{T}\colon \cB\to \bB(X)$ be a faithful non degenerate representation  and $\kappa\colon H\to \bC$ the trivial representation.
	By propositions \ref{prop: CHBB universal property} and \ref{prop: CHGB universal property}, lemma \ref{lem:wc induction and restriction} and theorem \ref{thm:Fells absorption principle I}, for all $f\in C^*(\cB)$ we have
	\begin{equation*}
	\|q^{H\uparrow\cB}(f)\|=\| \Ind_H^\cB(T|_{\cB_H}\otimes \kappa)_f \|\preceq \|[T\otimes \Ind_H^G(\kappa)]_f\|\leq \|q^{H\uparrow G}_\cB(f)\|;
	\end{equation*}
	which implies the existence of $\Theta^H_\cB.$
	Uniqueness is a consequence of the fact that $q^{H\uparrow G}_\cB$ is surjective.
	In case $\cB$ is saturated, $\Theta^H_\cB$ is the inverse of the $\pi^H_\cB$ of \eqref{equ: pi uparrow H}.
	
	To prove the claim about the intersection of the kernels we take a faithful non degenerate representation $\intform{U}$ of $H.$
	The integrated form of all the restrictions $T|_{\cB_{rH\rmu}}$ and conjugations ${}_r U$ have faithful integrated forms, so for all $f\in C^*(\cB)$ we have
	\begin{align*}
	q^{H\uparrow G}_\cB(f)\in \bigcap_{r\in G} \ker(\Theta^{rH\rmu}_\cB) &
	\Leftrightarrow q^{rH\rmu \uparrow \cB}_\cB(f)\ \forall\ r\in G \nonumber \\
	&\Leftrightarrow \| \Ind_{rH\rmu}^\cB(T|_{\cB_{rH\rmu}}\otimes {}_r U)_f \|=0\ \forall\ r\in G\nonumber  \\
	& \Leftrightarrow \|[T\otimes \Ind_H^G(U)]_f\|=0 \nonumber\\
	& \Leftrightarrow q^{H\uparrow G}_\cB(f)=0,
	\end{align*}
	which completes the proof.
\end{proof}

In many parts of \cite{FlDr88} Fell gets a result for representations of Banach *-algebraic bundles out of the corresponding result for Fell bundles.
Such techniques yield the following.

\begin{corollary}\label{thm:Fells absorption principle II}
	Let $\cB=\{B_t\}_{t\in G}$ be a Banach *-algebraic bundle over a LCH group, $H\leqslant G,$ $T\colon \cB\to \bB(X)$ a non degenerate representation and $U\colon H\to \bB(Y)$ a unitary representation.
	If for each $t\in G$  we denote ${}_tU$ the conjugated representation $tH\tmu \to \bB(Y), \ r\mapsto U_{\tmu r t},$ then $\{T|_{\cB_{tH\tmu}}\otimes {}_tU\}_{t\in G}$ is a set of $\cB-$positive representations  and 
	\begin{equation}\label{equ:thesis Fell for banach algebraic bundles}
	T\otimes \Ind_H^G(U)\approx \{\Ind_{tH\tmu }^{\cB}(T|_{\cB_{tH\tmu}}\otimes {}_tU) \}_{t\in G}.
	\end{equation}
\end{corollary}
\begin{proof}
	Let $\cC$ be the bundle C*-completion of $\cB$ and let $\rho\colon \cB\to \cC$ be the canonical quotient map of \cite[VIII 16.7]{FlDr88}.
	The construction of $\cC$ implies the existence of a unique representation $S\colon \cC\to \bB(X)$ such that $S\circ \rho = T.$
	
	Theorem \ref{thm:positivity} implies $S|_{\cC_{tH\tmu}}\otimes {}_t U $ is $\cC-$positive for every $t\in G.$
	Hence, by \cite[XI 12.6]{FlDr88}, the composition  $(S|_{\cC_{tH\tmu}}\otimes {}_t U)\circ (\rho|_{\cB_{tH\tmu}})\equiv T|_{\cB_{tH\tmu}}\otimes {}_tU$ is $\cB-$positive for all $t\in G.$
	
	We know that $S\otimes \Ind_H^G(U)\approx \{\Ind_{tH\tmu }^{\cC}(S|_{\cC_{tH\tmu}}\otimes {}_tU) \}_{t\in G},$ which implies 
	\begin{equation}\label{equ:almost there}
	(S\otimes \Ind_H^G(U))\circ \rho \approx \{\Ind_{tH\tmu }^{\cC}(S|_{\cC_{tH\tmu}}\otimes {}_tU)\circ \rho \}_{t\in G}.
	\end{equation}
	
	It is clear that $(S\otimes \Ind_H^G(U))\circ \rho = (S\circ \rho)\otimes \Ind_H^G(U) = T\otimes \Ind_H^G(U).$
	On the other hand, by \cite[XI 12.6]{FlDr88}, for all $t\in G$ we have
	\begin{align*}
	\Ind_{tH\tmu }^{\cC}(S|_{\cC_{tH\tmu}}\otimes {}_tU)\circ \rho 
	& = \Ind_{tH\tmu }^{\cB}\left( (S|_{\cC_{tH\tmu}}\otimes {}_tU)\circ (\rho|_{\cB_{tH\tmu}})\right) \\
	& =\Ind_{tH\tmu }^{\cB}( T|_{\cB_{tH\tmu}}\otimes {}_tU).
	\end{align*}
	Thus \eqref{equ:almost there} implies \eqref{equ:thesis Fell for banach algebraic bundles}.
\end{proof}

\subsection{Exel-Ng's absorption principle}

If the subgroup $H$ of theorem \ref{equ: action to induce} is normal in $G,$ then all the restrictions $T|_{\cB_{tH\tmu}}$ become $T|_{\cB_H}$ and $T\otimes\Ind_H^G(U)\approx \{\Ind_H^\cB(T|_{\cB_H}\otimes {}_tU)\}_{t\in G}.$
One can arrange $U$ to have ${}_tU\approx U$ for all $t\in G$.
Indeed, this is the case if $\intform{U}$ is faithful or $U:=\bigoplus_{s\in G} {}_s V$ for some other representation $V$ of $H.$

\begin{corollary}\label{cor:normal subgroup}
	If, in addition to the hypotheses of theorem~\ref{equ: action to induce}, we assume that $H$ is normal in $G$ and ${}_tU\approx U$ for all $t\in G,$ then $T\otimes\Ind_H^G(U)\approx \Ind_H^\cB(T|_{\cB_H}\otimes U).$
\end{corollary}
\begin{proof}
	It is clear that $\Ind_H^\cB(T|_{\cB_H}\otimes U)\preceq \{\Ind_{tH\tmu }^{\cB}(T|_{\cB_{tH\tmu}}\otimes {}_tU) \}_{t\in G}.$
	Besides, for any $t\in G$ we have $T|_{\cB_H}\otimes {}_t U\preceq T|_{\cB_H}\otimes U$ and this implies 
	\begin{equation*}
	\Ind_H^\cB(T|_{\cB_H}\otimes {}_t U)\preceq \Ind_H^\cB(T|_{\cB_H}\otimes U)
	\end{equation*}
	because induction preserves weak containment.
	Then $\Ind_H^\cB(T|_{\cB_H}\otimes U)\approx \{\Ind_{tH\tmu }^{\cB}(T|_{\cB_{tH\tmu}}\otimes {}_tU) \}_{t\in G}$ and the thesis follows from theorem~\ref{equ: action to induce} (and the transitivity of weak equivalence).
\end{proof}

The hypotheses of the corollary above are fulfilled if $T\colon \cB\to \bB(X)$ is any representation, $H=\{e\}$ and $U\colon H\to \bC,\  s\mapsto 1.$
In this case corollary \ref{cor:normal subgroup} becomes Exel-Ng's absorption principle.

\section{Weak containment and subgroups}

In \cite{ExNg} a Fell bundle $\cB$ is called amenable if $\regrepEN{\cB}\colon C^*(\cB)\to \bB(L^2_e(\cB))$ is faithful, which we write\footnote{Recall the convention introduced at the beginning of section \ref{sec:positivity}.} $C^*(\cB)=C^*_\red(\cB).$
This is equivalent to say any representation $T$ of $\cB$ is weakly contained in some other of the form $S\otimes \lambda$ (this is the so called \textit{weak containment property}, WCP).

The recent developments on amenable actions (\'a la Anantharaman-Delaroche) suggest Exel and Ng should have named \textit{amenable} those bundles having the approximation property they introduced in \cite{ExNg}.
We follow this stream, so amenability implies the WCP.

The reduced representation $\regrepEN{\cB},$ considered as a map from $C^*_{G\uparrow \cB}(\cB)=C^*_{G\uparrow G}(\cB)=C^*(\cB)$ to $C^*_{\{e\}\uparrow \cB}(\cB)=C^*_{\{e\}\uparrow G}(\cB)=C^*_\red(\cB),$ is a particular case of the maps $\mu$ and $\nu $ of

\begin{proposition}\label{prop: maps mu and nu}
	Let $\cB=\{B_t\}_{t\in G}$ be a Fell bundle.
	Given subgroups $K \leqslant H\leqslant G$ there exist unique *-homomorphisms $\mu^{HK}_\cB\colon C^*_{H\uparrow \cB}(\cB)\to C^*_{K\uparrow \cB}(\cB)$ and $\nu^{HK}_\cB\colon C^*_{H\uparrow G}(\cB)\to C^*_{K\uparrow G}(\cB)$ such that $\mu^{HK}_\cB\circ q^{H\uparrow \cB}=q^{K\uparrow \cB}$ and $\nu^{HK}_\cB\circ q^{H\uparrow G}_\cB=q^{K\uparrow G}_\cB.$
\end{proposition}
\begin{proof}
	The existence of $\mu^{HK}_\cB$ is equivalent to the validity of $\| q^{K\uparrow \cB}(f) \|\leq \| q^{H\uparrow \cB}(f) \|$ for all $f\in C^*(\cB).$
	To prove this we take a faithful non degenerate  representation $\intform{T}\colon C^*(\cB_K)\to \bB(X).$
	By proposition \ref{prop: CHBB universal property} and induction in stages \cite[XI 12.15]{FlDr88},
	\begin{equation}\label{equ: existence of mu HKB}
	\| q^{K\uparrow \cB}(f) \|=\|\intform{\Ind}_H^\cB(T)_f\|=\|\intform{\Ind}_K^\cB(\Ind_H^{\cB_K}(T))_f\|\leq \|q^{H\uparrow \cB}(f)\|.
	\end{equation}
	
	The proof of the existence of $\nu^{HK}_\cB$ is similar. 
	It combines induction in stages with proposition \ref{prop: CHGB universal property}, the details are left to the reader with the suggestion to consult \eqref{equ: nuHK B isometric}.
\end{proof}

\begin{corollary}
	The restrictions of both $q^{H\uparrow \cB}$ and $q^{H\uparrow G}_\cB$ to $L^1(\cB)$ are faithful.
\end{corollary}
\begin{proof}
	We have $ \mu^{H\{e\}}_\cB \circ q^{H\uparrow\cB} = q^{\{e\}\uparrow \cB} = \regrepEN{\cB}=q^{\{e\}\uparrow G}_\cB = \nu^{H\{e\}}_\cB \circ q^{H\uparrow G}_\cB,$ so it suffices to consider the case $H=\{e\}.$
	In \cite[VIII 16.4]{FlDr88} Fell proves that the direct sum of the integrated forms of the generalized regular representations of $\cB$ is faithful.
	This is equivalent to say that $\intform{\Ind}_{\{e\}}^\cB(\pi)|_{L^1(\cB)}=\regrepEN{\cB}\otimes_\pi 1|_{L^1(\cB)}$ is faithful, with $\pi$ the universal representation of $B_e.$
	Thus $\regrepEN{\cB}|_{L^1(\cB)}$ is faithful.      
\end{proof}

\begin{corollary}\label{cor:mu and nu properties}
	In the situation of proposition \ref{prop: maps mu and nu}, the following claims hold:
	\begin{enumerate}
		\item\label{item: WCP and maps mu} $C^*(\cB)=C^*_\red(\cB)$ $\Leftrightarrow$ $C^*_{H\uparrow \cB}(\cB)=C^*_\red(\cB)$ for all $ H\leqslant G$ $\Leftrightarrow$ $C^*_{H\uparrow G}(\cB)=C^*_\red(\cB)$ for all $H\leqslant G.$
		\item\label{item: muHK isomorphism} If $C^*(\cB_H)=C^*_{K\uparrow \cB_H}(\cB_H),$ then  $C^*_{K\uparrow\cB}(\cB)=C^*_{H\uparrow \cB}(\cB).$
		\item If $C^*(H)=Q^H_K,$ then  $C^*_{H\uparrow G}(\cB)=C^*_{K\uparrow G}(\cB).$
		\item If $H$ is amenable, then $C^*_{H\uparrow \cB}(\cB)=C^*_\red(\cB)=C^*_{H\uparrow G}(\cB).$
	\end{enumerate}
\end{corollary}
\begin{proof}
	The first claim follows easily after one notices  $\mu^{G\{e\}}_{\cB} =\mu^{H\{e\}}_{\cB} \circ \mu^{GH}_{\cB};$ $\nu^{G\{e\}}_{\cB} =\nu^{H\{e\}}_{\cB} \circ \nu^{GH}_{\cB}$  and identifies $\mu^{G\{e\}}_{\cB}=\regrepEN{\cB}=\nu^{G\{e\}}_{\cB}.$
	For the second claim we go back to \eqref{equ: existence of mu HKB}.
	The hypothesis implies the integrated form of $\Ind_H^{\cB_K}(T)$ is faithful, so the inequality of \eqref{equ: existence of mu HKB} becomes an equality.
	
	The proof of the third claim is similar.
	Take faithful non degenerate representations $\intform{T}\colon C^*(\cB)\to \bB(X)$ and $\intform{U}\colon C^*(K)\to \bB(Y).$
	Then $\intform{\Ind}_K^H(U)$ is faithful and proposition \ref{prop: CHGB universal property} (together with induction in stages) implies that for all $f\in C^*(\cB)$
	\begin{equation}\label{equ: nuHK B isometric}
	\| q^{K\uparrow \cB}(f) \| = \|[T\intform{\otimes}\Ind_K^G(U)]_f\|
	= \|[T\intform{\otimes}\Ind_H^G(\Ind_K^H(U))]_f\|=\| q^{H\uparrow \cB}(f) \|,
	\end{equation}
	which clearly implies $\nu^{HK}_\cB$ is isometric.  
	
	If $H$ is amenable and we set $K=\{e\},$ then $C^*_{K\uparrow\cB}(\cB)=C^*_{K\uparrow G}(\cB)=C^*_\red(\cB);$ $C^*(H)=C^*_\red(H)=Q^H_K$ and  $C^*(\cB_H)= C^*_\red(\cB_H)$ (see \cite{ExNg}).
	Thus the last claim is a consequence of the preceding ones.
\end{proof}

The converse of claim \eqref{item: muHK isomorphism} of the corollary above is intimately related to question \eqref{item: q universal algebras agree} of the introduction.
If we put $K=\{e\}$ in the claim, we get $C^*(\cB_H)= C^*_\red(\cB_H)$ implies  $C^*_{H\uparrow\cB}(\cB)=C^*_\red(\cB).$
Whenever the converse holds, one can use claim \eqref{item: WCP and maps mu} of corollary \ref{cor:mu and nu properties} to deduce the WCP passes from $\cB$ to $\cB_H.$
This is false in general because $\cB$ may be the semidirect product bundle of a C*-dynamical system $(A,G,\alpha),$ in which case $\cB_H$ is that of the restricted system  $(A,H,\alpha|_H).$
By \cite[Section 5.3]{buss2020amenability}, it is not true that $A\rtimes_\alpha G = A\rtimes_{\red\alpha} G$ (i.e. $C^*(\cB)=C^*_\red(\cB)$) implies $A\rtimes_{\alpha|_H} H = A\rtimes_{\red\alpha|_H} H$ ($C^*(\cB_H)=C^*_\red(\cB_H)$).

In theorem \ref{thm:equivalence of weak containment of reduction} we show that whenever the normalizer of $H$ is open in $G,$ $\cB_H$ has the WCP if and only if $C^*_{H\uparrow\cB}(\cB)=C^*_\red(\cB).$
When applied to semidirect product bundles, this gives a condition on $H$ for the WCP to pass from $(A,G,\alpha)$ to $(A,H,\alpha|_H)$ (see corollary \ref{cor:WCP passes to certain subgroups}); giving a partial affirmative answer to Question (b) of \cite[Section 9]{ADamenability2002}.

We need a (probably well known) fact about the regular representations of subgroups.

\begin{proposition}\label{prop:weak equivalence regular rep}
	Let $G$ be a LCH group and $H$ a closed  subgroup of $G$ with open normalizer.
	If we denote $\lambda^G$ the left regular representation of $G,$ then $\lambda^G|_H\approx\lambda^H.$
\end{proposition}
\begin{proof}
	Assume $H$ is open in $G$ and decompose $L^2(G)$ into a direct sum with respect to the right cosets $H\setminus G=\{Ht\colon t\in G\},$ that is $L^2(G)=\bigoplus_{Z\in H\setminus G} L^2(Z).$
	With this decomposition $\lambda^G|_H$ becomes a direct sum of representations unitary equivalent to $\lambda^H,$ so $\lambda^G|_H\approx \lambda^H.$
	
	We now assume that $H$ is normal in $G$ and that the left invariant Haar measures of $G,H$ and $G/H$ have been normalized so that $	\int_G f(t)\, dt = \int_{G/H} \int_H f(ts)\, dsd(tH)$
	for all $f\in C_c(G)$ (exactly as in \cite[III 13.17]{FlDr88}).
	Besides, we let $\Gamma\colon G\to \bR^+$ be the continuous homomorphism of \cite[III 13.20]{FlDr88}, i.e. the unique function such that $\int_H f(sts^{-1})\, dt =\Gamma(s)\int_H f(t)\, dt $ for all $f\in C_c(H)$ and $s\in G.$

	For any $f\in C_c(H)$ and $\xi\in C_c(G)\subset L^2(G)$ we have
	\begin{align}
	\| \intform{(\lambda^G|_H)}_f \xi \|^2
	& =\int_G\int_H f^**f(t)\xi(\tmu s)\overline{\xi(s)}dtds\nonumber \\
	& = \int_{G/H} \int_H\int_H f^**f(t)\xi(\tmu sr)\overline{\xi(sr)}\, dtdrd(sH) \nonumber \\
	& =	\int_{G/H} \Gamma(s)\int_H\int_H f^**f(t)\xi(\tmu rs)\overline{\xi(rs)}dtdrd(sH).\label{equ:bound reg rep}
	\end{align}
	
	To bound the inner double integral in the last term above we define, for each $s\in G,$ $\xi_s\in C_c(H)\subset L^2(H)$ by $\xi_s(z)=\xi(zs).$
	Then 
	\begin{equation*}
	\int_H\int_H f^**f(t)\xi(\tmu rs)\overline{\xi(rs)}dtdr
	=\int_H f^**f(t)\langle \lambda^H_t \xi_s,\xi_s\rangle \, dt
	\leq \|\intform{\lambda^H}_f\|^2\|\xi_s\|^2
	\end{equation*}
	and we can continue \eqref{equ:bound reg rep} to get
	\begin{equation*}
	\| \intform{(\lambda^G|_H)}_f \xi \|^2
	\leq \|\intform{\lambda^H}_f\|^2 \int_{G/H} \int_H \xi(sr)\overline{\xi(sr)}\, drd(sH)
	= \|\intform{\lambda^H}_f\|^2\|\xi\|^2.
	\end{equation*}
	Thus $\| \intform{(\lambda^G|_H)}_f\|\leq \|\intform{\lambda^H}_f\|$ for all $f\in C_c(H)$ and this implies $\lambda^G|_H\preceq \lambda^H.$
	Lemma \ref{lem:wc induction and restriction} gives $\lambda^H\preceq \lambda^G|_H,$ so $\lambda^H\approx \lambda^G|_H.$
	
	For the general case we define $N$ as the normalizer of $H$ in $G$ and use the arguments of the proof of lemma \ref{lem:wc induction and restriction} to get $
	\lambda^G|_H =\lambda^G|_N|_H\approx \lambda^N|_H\approx \lambda^H,$ which implies $\lambda^G|_H\approx \lambda^H.$
\end{proof}

\begin{theorem}\label{thm:equivalence of weak containment of reduction}
	Let $\cB=\{B_t\}_{t\in G}$ be a Fell bundle and $H\leqslant G.$
	If the normalizer of $H$ is open, then $C^*_{H\uparrow \cB }(\cB)=C^*_\red(\cB)$ if and only if $C^*(\cB_H)=C^*_\red(\cB_H)$ and these conditions hold if $C^*_{H\uparrow G}(\cB)=C^*_\red(\cB).$
\end{theorem}
\begin{proof}
	We have $\mu^{H\{e\}}_\cB\circ \Theta^H_\cB=\nu^{H\{e\}}_\cB.$
	Hence, $C^*_{H\uparrow G}(\cB)=C^*_\red(\cB)$ (i.e. $\nu^{H\{e\}}_\cB$ an isomorphism) if and only if $C^*_{H\uparrow G}(\cB)=C^*_{H\uparrow \cB}(\cB)=C^*_\red(\cB)$ (both $\mu^{H\{e\}}_\cB$ and $\Theta^H_\cB$ are isomorphisms).
	By corollary \ref{cor:mu and nu properties}, $C^*(\cB_H)=C^*_\red(\cB_H)$ implies $C^*_{H\uparrow \cB }(\cB)=C^*_\red(\cB).$
	
	Suppose $C^*_{H\uparrow \cB}(\cB)=C^*_\red(\cB)$ and take a faithful non degenerate representation $\intform{T}$ of $\cB.$
	By proposition \ref{prop: CHBB universal property}, $\Ind_H^\cB (T|_{\cB_H})\approx \Ind_{\{e\}}^\cB (T|_{B_e})$ and this, together with lemma \ref{lem:wc induction and restriction}, yields $T|_{\cB_H}\preceq \Ind_H^\cB (T|_{\cB_H})|_{\cB_H}\approx \Ind_{\{e\}}^\cB (T|_{B_e})|_{\cB_H}$ $\Rightarrow$ $T|_{\cB_H}\preceq \Ind_{\{e\}}^\cB (T|_{B_e})|_{\cB_H}.$
	By Exel-Ng's absorption principle, $\Ind_{\{e\}}^\cB (T|_{B_e})|_{\cB_H}\approx T|_{\cB_H}\otimes \lambda^G|_H$ and from lemma \ref{prop:weak equivalence regular rep} we get $\Ind_{\{e\}}^\cB (T|_{B_e})|_{\cB_H}\approx T|_{\cB_H}\otimes \lambda^H$ and we conclude $T|_{\cB_H}\preceq T|_{\cB_H}\otimes \lambda^H.$
	Lemma \ref{lem:wc induction and restriction} implies the integrated for of $T|_{\cB_H}$ is faithful, thus $\cB_H$ has the WCP.
\end{proof}

\begin{corollary}\label{cor:WCP passes to certain subgroups}
	If $\cB=\{B_t\}_{t\in G}$ is a Fell bundle, $C^*(\cB)=C^*_\red(\cB)$ and the normalizer of $H\leqslant G$ is open, then $C^*(\cB_H)=C^*_\red(\cB_H).$
\end{corollary}
\begin{proof}
	This is a straightforward consequence of corollary \ref{cor:mu and nu properties} and theorem \ref{thm:equivalence of weak containment of reduction}.
\end{proof}

\bibliography{/home/damian/Nextcloud/Bibliografia/FerraroBiblio}

\begin{thebibliography}{10}

\bibitem{Ab03}
Fernando Abadie.
\newblock Enveloping actions and {T}akai duality for partial actions.
\newblock {\em J. Funct. Anal.}, 197(1):14--67, 2003.

\bibitem{AbFrrEquivalence}
Fernando Abadie and Dami\'{a}n Ferraro.
\newblock Equivalence of {F}ell bundles over groups.
\newblock {\em J. Operator Theory}, 81(2):273--319, 2019.

\bibitem{ADamenability2002}
Claire Anantharaman-Delaroche.
\newblock Amenability and exactness for dynamical systems and their
  {$C^\ast$}-algebras.
\newblock {\em Trans. Amer. Math. Soc.}, 354(10):4153--4178, 2002.

\bibitem{blattner1961induced}
Robert~J Blattner.
\newblock On induced representations.
\newblock {\em American Journal of Mathematics}, 83(1):79--98, 1961.

\bibitem{buss2020amenability}
Alcides Buss, Siegfried Echterhoff, and Rufus Willett.
\newblock Amenability and weak containment for actions of locally compact
  groups on {$C^\ast$}-algebras.
\newblock {\em arXiv preprint arXiv:2003.03469}, 2020.

\bibitem{Derighetti}
Antoine Derighetti.
\newblock Some remarks on {$L^1(G)$}.
\newblock {\em Math. Z.}, 164(2):189--194, 1978.

\bibitem{ExlibroAMS}
Ruy Exel.
\newblock {\em Partial dynamical systems, Fell bundles and applications},
  volume 224 of {\em Mathematical Surveys and Monographs}.
\newblock American Mathematical Soc., Providence, RI, 2017.

\bibitem{ExNg}
Ruy Exel and Chi-Keung Ng.
\newblock Approximation property of {$C^\ast$}-algebraic bundles.
\newblock {\em Mathematical Proceedings of the Cambridge Philosophical
  Society}, 132:509--522, 5 2002.

\bibitem{Fell}
J.~M.~G. Fell.
\newblock {\em Induced representations and {B}anach {$*$}-algebraic bundles}.
\newblock Lecture Notes in Mathematics, Vol. 582. Springer-Verlag, Berlin-New
  York, 1977.
\newblock With an appendix due to A. Douady and L. Dal Soglio-H\'{e}rault.

\bibitem{FlDr88}
J.~M.~G. Fell and R.~S. Doran.
\newblock {\em Representations of {$^*$}-algebras, locally compact groups, and
  {B}anach {$*$}-algebraic bundles.}, volume 125--126 of {\em Pure and Applied
  Mathematics}.
\newblock Academic Press Inc., Boston, MA, 1988.

\bibitem{KaLAQu2013}
S.~Kaliszewski, Magnus~B. Landstad, and John Quigg.
\newblock Exotic group {$C^\ast$}-algebras in noncommutative duality.
\newblock {\em New York J. Math.}, 19:689--711, 2013.

\bibitem{Mk52}
George~W. Mackey.
\newblock Induced representations of locally compact groups. {I}.
\newblock {\em Ann. of Math. (2)}, 55:101--139, 1952.

\bibitem{mckee2020amenable}
Andrew McKee and Reyhaneh Pourshahami.
\newblock Amenable and inner amenable actions and approximation properties for
  crossed products by locally compact groups.
\newblock {\em Canadian Mathematical Bulletin}, pages 1--19, 2020.

\bibitem{Raeburn1992Crossed}
Iain Raeburn.
\newblock On crossed products by coactions and their representation theory.
\newblock {\em Proc. London Math. Soc. (3)}, 64(3):625--652, 1992.

\bibitem{Raeburn1998morita}
Iain Raeburn and Dana~P. Williams.
\newblock {\em Morita equivalence and continuous-trace {$C^\ast$}-algebras}.
\newblock Number~60. American Mathematical Soc., 1998.

\bibitem{Rf74}
Marc~A. Rieffel.
\newblock Induced representations of {$C^\ast$}-algebras.
\newblock {\em Advances in Math.}, 13:176--257, 1974.

\end{thebibliography}
\bibliographystyle{plain}
\end{document}